\documentclass[noinfoline]{imsart}
\RequirePackage[OT1]{fontenc}
\RequirePackage{amsthm,amsmath,natbib}
\setcitestyle{square,numbers}
\RequirePackage[colorlinks,citecolor=blue,urlcolor=blue]{hyperref}
\usepackage{enumerate}
\usepackage{amsmath}
\usepackage{amsfonts}

\usepackage{geometry}
\newtheorem{theorem}{Theorem}
\newtheorem{lemma}[theorem]{Lemma}
\newtheorem{proposition}[theorem]{Proposition}
\newtheorem{corollary}[theorem]{Corollary}
\theoremstyle{definition}
\newtheorem{definition}[theorem]{Definition}

\theoremstyle{remark}
\newtheorem{remark}[theorem]{Remark}
\numberwithin{equation}{section} 

\newcommand{\Real}{\mathbb{R}}

\newcommand{\R}{\mathbb{R}}

\newcommand{\law}{\operatorname{Law}}

\begin{document}

\begin{frontmatter}
\title{Mean-field limit for a class of stochastic ergodic control problems}
\runtitle{Mean-field limit}

\begin{aug}
\author{\fnms{Sergio} \snm{Albeverio}\thanksref{a,e1}\ead[label=e1,mark]{albeverio@iam.uni-bonn.de}}, \author{\fnms{Francesco C.} \snm{De Vecchi}\thanksref{a,e2}\ead[label=e2,mark]{francesco.devecchi@uni-bonn.de}}, \author{\fnms{Andrea} \snm{Romano}\thanksref{b,e3}%
\ead[label=e3,mark]{andrea.romano4@studenti.unimi.it}}
\and \author{\fnms{Stefania} \snm{Ugolini}\thanksref{b,e4}%
\ead[label=e4,mark]{stefania.ugolini@unimi.it}}

\affiliation{Rheinische Friedrich-Wilhelms-Universit\"at Bonn. and Universit\`a degli Studi di Milano}

\address[a]{Institute for Applied Mathematics and Hausdorff Center for Mathematics, 
Endenicher Allee 60, 53115 Bonn, Germany.
\printead{e1,e2}}

\address[b]{Dipartimento di Matematica, 
 Via Saldini 50, 20113 Milano, Italy.
\printead{e3,e4}}

\runauthor{Albeverio S. et al.}

\end{aug}

\begin{abstract}
We study a family of McKean-Vlasov (mean-field) type ergodic optimal control problems with linear control, and quadratic dependence on control of the cost function. For this class of problems we establish existence and uniqueness of an optimal control. We propose an $N$-particles Markovian optimal control problem approximating the McKean-Vlasov one and we prove the convergence in relative entropy, total variation and Wasserstein distance of the law of the former to the law of the latter when $N$ goes to infinity. Some McKean-Vlasov optimal control problems with singular cost function and the relation of these problems with the mathematical theory of Bose-Einstein condensation is also established.
\end{abstract}


\end{frontmatter}

\section{Introduction}

In this paper we want to provide a complete study of a family of mean-field ergodic stochastic optimal control problems, known as optimally controlled McKean-Vlasov dynamics. More precisely we consider the controlled stochastic differential equation (SDE)
\begin{equation}
dX_t=\alpha(X_t)dt+\sqrt{2}dW_t \label{eq:SDEmain}
\end{equation}
where $\alpha$ is a $C^1$ control function from  $\mathbb{R}^n$ to  $\mathbb{R}^n$ and $W_t$, $t \geq 0$, is an $n$ dimensional standard Brownian motion, with the following cost functional 
\begin{equation}\label{eq:J1}
J(\alpha,x_0)=\limsup_{T \rightarrow +\infty}\frac{1}{T} \left(\int_0^T\mathbb{E}_{x_0}\left[\frac{|\alpha(X_t)|^2}{2}+\mathcal{V}(X_t,\law(X_t)) \right]dt \right).
\end{equation}
Here $\mathcal{V}:\mathbb{R}^n \times \mathcal{P}(\mathbb{R}^n) \rightarrow \mathbb{R}$ (where $\mathcal{P}(\mathbb{R}^n)$ is the space of probability measures on $\mathbb{R}^n$ endowed with the metric given by the weak convergence) is a regular function satisfying some technical hypotheses (see Section \ref{section_technical} below) and $\mathbb{E}_{x_0}$ is the expectation with respect to the solution $X_t$ to the SDE \eqref{eq:SDEmain} such that $X_0=x_0\in \mathbb{R}^n$. We prove  existence and uniqueness of the optimal control $\alpha\in C^1(\mathbb{R}^n,\mathbb{R}^n)$ for the problem given by \eqref{eq:SDEmain} and \eqref{eq:J1}. Furthermore we present an $N$-particle Markovian approximation of the previous problem and give a proof of the convergence of the corresponding value functions and of the invariant finite dimensional probability laws to those given by the one of the McKean-Vlasov dynamics.  We remark that the cost functional explicitly depends on the law of $X_t$ and, under natural assumptions, we prove that  the optimal control can also be expressed in terms of the same law.\\ 

Recently there has been a growing interest in optimally controlled McKean-Vlasov dynamics (see, for example, \cite{Basei,Cosso,Bensoussan,Carmona,Carmonabook1,Pham,Pham1}). The main part ot the current literature focuses on finite or infinite time horizon problems and usually does not discuss the approximation of the controlled McKean-Vlasov problem by  Markovian controlled $N$-particle systems. To the best of our knowledge some of the few exceptions are \cite[Chapter 6]{Carmonabook2} (see also \cite{Carmona}), where, for the case of controlled McKean-Vlasov dynamics, the convergence of the value function is considered (which implies the convergence of the optimal trajectory see again \cite[Chapter 6]{Carmonabook2}), and \cite{Lacker1} and \cite{Djete2020mckeanvlasov,Djete2020extended}. In particular \cite{Lacker1} studies the convergence problem under general conditions, without symmetry assumptions and in the time-dependent setting using a martingale problem approach. In \cite{Djete2020mckeanvlasov} the techniques of \cite{Lacker1} are generalized to the case of controlled McKean-Vlasov problems having generally dependent noises, and in \cite{Djete2020extended} a propagation of chaos result is proven for \emph{extended} mean field control problems.\\
The optimally controlled McKean-Vlasov dynamics is closely related to mean-field games theory (see \cite[Chapter 6]{Carmonabook1} for a discussion about the relation between the two approaches). Mean-field games theory, in the case of a finite and infinite time horizon utility function, is much more developed both in the study of the limit problem and in the study of $N$-particles approximations (see, e.g., the books \cite{Lions2019,Carmonabook1,Carmonabook2} and references therein as well \cite{Hongler2020}). The PDE system related to the ergodic mean-field games is well studied (see, e.g., \cite{Lions2013,Porretta,Lions2012,Cirant2015,Lions2006}). In the mean-field games case, the ergodic stochastic problem is considered in \cite{Ergodicpaper,Bardi2016,Bardi2014,Feleqi2013}, see also \cite{Cardaliaguet2020}. \\
Our convergence scheme is quite different from the one usually formulated in the literature on competitive mean-field games, where the value function can be decomposed into the product of the one-particle marginals (see \cite{Ergodicpaper,Lions2006}). Indeed our $N$-particles process is an interacting controlled diffusions system where the chaoticity property is  achieved only asymptotically (that is in the infinite particles limit). It is important to note that, as it is proved in \cite{Cardaliaguet2020}, in the ergodic competitive mean-field game case, when the control depends on all $N$-players, the convergence of the $N$-particle system to the mean-field one in general does not hold. This is one of the main differences with respect to the (cooperative) McKean-Vlasov systems treated in the present paper. To the best of our knowledge, this is the first paper facing  in an ergodic framework the convergence problem of a Markovian $N$ interacting diffusions system to a Markovian limit system of McKean-Vlasov type.\\
The main idea of the paper is to exploit some methods of mathematical physics, in particular from the mathematical theory of Bose-Einstein condensation (see, e.g.\cite{Lewin2,Lewin,Lieb2,LiebBook,Lieb,Nam,Rougerie,rougerie2021scaling}), Nelson's stochastic mechanics (see, e.g., \cite{Carlen1,Carlen2,GuerraMorato,Nelson1,Yasue1}), and variational stochastic processes connected with Schr\"odinger problem in optimal transport and the Hopf-Cole transformation (see, e.g., \cite{Leonard2020,Pavon2016,cruzeiro2020time,Zambrini2016,Leonard2014,Mikami2006,Zambrini1986}  see also \cite{Cirant2015,Ullmo2019}). Our method uses however a compactness argument, hence it does not yield per se convergence rates of optimal trajectories (such convergence results exist in other settings, but under the stronger assumption of convexity, see \cite{Carmona,Carmonabook1}).\\

The paper contains three main results. The first one is the proof of existence and uniqueness of the optimal control for the SDE \eqref{eq:SDEmain} with cost functional \eqref{eq:J1} under the technical Hypotheses $\mathcal{V}$ (concerning the functional $\mathcal{V}$ in equation \eqref{eq:J1}) and a convexity request C$\mathcal{V}$ for the cost functional discussed in Section \ref{section_technical}. We show that some results (\cite{Rockner}) guarantee that under appropriate conditions the control term is the logarithmic derivative of the probability density of the process and so the cost functional can be expressed in terms of the process probability density. By exploiting calculus of variations we then provide a necessary condition for the optimality of the process probability density (Theorem \ref{theorem_optimalcontrol1}). The method applied here makes no direct use of Hamilton-Jacobi-Bellman equation.\\
Our second main result consists in the convergence of the value function (or rather the constant which gives the value of the cost functional evaluated at the optimal control) and of the finite dimensional invariant distributions of the Markovian $N$-particle approximation to the one of the McKean-Vlasov optimal control problem when the number of particles $N$ tends to $ + \infty$ (see Theorem \ref{theorem_VE}, Theorem \ref{theorem_HH1} and Remark \ref{remark_HH1}). The convergence of the value function, under the previous Hypotheses $\mathcal{V}$ and C$\mathcal{V}$, is achieved using in an essential way de Finetti theorem for exchangeable particles and some important properties of Fisher information (see Section \ref{section_example}). The convergence of the finite dimensional distributions, under the additional quadratic growth Hypothesis Q$V$ (see Remark \ref{remark_QV} for an analysis of the role of Hypothesis Q$V$ in the proof of Theorem \ref{theorem_HH1}), is obtained by proving that the relative entropy between the finite dimensional distributions, converges to $0$ for $N \rightarrow +\infty$ (Theorem \ref{theorem_HH1}). The latter is achieved by exploiting some results from the theory of interacting particles systems and optimal transport (see \cite{HaMi}).\\
The third main result is the convergence in total variation of the law of the $N$-particles approximation to the law of the McKean-Vlasov system. In this way we also establish (see Theorem \ref{theorem_main1}) that the strong Kac's chaos holds for the probability law of the $N$-interacting controlled diffusions system in the limit of infinitely many particles (in the sense of \cite{Lacker}). Namely we prove that, under Hypotheses $\mathcal{V}$,  for any $k>0$, if $\mathbb{P}_{0,N}^{(k)}$ is the law (on path space) of the first $k$ particles of the optimal $N$-particles approximation and $\mathbb{P}_{0}$ is the law (on path space) of the optimal McKean-Vlasov system we have $\mathcal{H}(\mathbb{P}_{0,N}^{(k)}|\mathbb{P}_{0}^{\otimes k})\rightarrow 0$ as $N \rightarrow +\infty$ (where $\mathcal{H}(\cdot|\cdot)$ denotes the relative entropy of the first measure with respect to the second one). This kind of convergence implies the convergence in total variation (Corollary \ref{corollary_TV}) and in Wasserstein metric $W_p$ (for $1\leq p\leq 2$) (Corollary \ref{corollary_Wesserstein}). Let us remark that the result that is most closely related to our own are in \cite{Djete2020mckeanvlasov,Lacker1}, where a more general problem is treated in the finite time horizon case proving a convergence in Wasserstein metric. The ergodic case treated and the type of convergence proved in our paper are however new. Indeed, for example, differently from \cite{Lacker1}, we start the system at the invariant measure, we have to prove that the $N$-particles invariant measures converge to the limit one. Furthermore we think that the entropy convergence does not hold in the general setting considered in \cite{Lacker1}, since when the noise is multiplicative, with diffusion coefficients depending on the control or the law of the solution process, the law of the $N$-particles approximation and of the limit solution are not longer mutually absolutely continuous.\\ 

The plan of the paper is as follows. In Section \ref{section_technical} we define our class of ergodic McKean-Vlasov optimal stochastic control problems, making explicit all our hypotheses and providing a non trivial family of cost functions satisfying them. In Section \ref{section_existence} we prove existence and uniqueness of the optimal control for our problem. In Section \ref{section_approximation} we introduce the Markovian $N$-particles controlled system used to approximate the McKean-Vlasov dynamics. In Section \ref{section_example} we prove the convergence of the value function of the $N$-particles approximation to the one of the McKean-Vlasov problem, and in Section \ref{section_fixedtime} the convergence of the probability law at fixed time is discussed. In Section \ref{section_convergence} the process convergence result on the path space in the infinite particles limit is established. A comparison with the mathematical physics literature and a comment on the result for the case of singular potentials are provided in Section \ref{section_physics}.

\section{The setting and the hypotheses}\label{section_technical}

We consider controlled SDE given by \eqref{eq:SDEmain}, with the assumptions stated there.
Thus $X=(X_t)_{t\geq 0}$ is an $n$ dimensional process, $W$ is an $n$ dimensional Brownian motion and $\alpha(X_t)$ is the control process, with $\alpha:\mathbb{R}^n \rightarrow \mathbb{R}^n$  a $C^1$ function. We denote by $L_{\alpha}=\Delta +\alpha \cdot \nabla$ the generator associated with the equation \eqref{eq:SDEmain} and by $L^*_{\alpha}$ the adjoint of $L_{\alpha}$ with respect to the Lebesgue measure.\\

We take  a functional 
$$\mathcal{V}:\mathbb{R}^n \times \mathcal{P}(\mathbb{R}^n) \rightarrow \mathbb{R}, $$
where $\mathcal{P}(\mathbb{R}^n)$ is the set of probability measures on $\mathbb{R}^n$. We also define for any $\mu \in \mathcal{P}(\mathbb{R}^n)$ (such that $\mathcal{V}(\cdot,\mu)$ if $\mu$ integrable)
\begin{equation}\label{functionalofmu}
\tilde{\mathcal{V}}(\mu):=\int_{\mathbb{R}^n}{\mathcal{V}(x,\mu)\mu(dx)}.
\end{equation} 
If $\mathcal{K}:\mathcal{P}(\mathbb{R}^n) \rightarrow \mathbb{R}$ is a function we say that $\mathcal{K}$ is G\^{a}teaux differentiable if for any $\mu, \mu' \in \mathcal{P}(\mathbb{R}^n)$ there exists a bounded continuous function  $\partial_{\mu}\mathcal{K}(\cdot,\mu):\mathbb{R}^n \rightarrow \mathbb{R}$ such that   
\begin{equation}\label{eq:derivative1}
\lim_{\epsilon \rightarrow 0^+}\frac{\mathcal{K}(\mu+\epsilon(\mu'- \mu))-\mathcal{K}(\mu)}{\epsilon}=\int_{\mathbb{R}^n}{\partial_{\mu}\mathcal{K}(y,\mu)(\mu'(dy)-\mu(dy))},
\end{equation}
(the left hand side of \eqref{eq:derivative1} is well defined, when the limit is defined, since, when $\epsilon \geq 0$, $\mu+\epsilon(\mu'- \mu)$ is a probability measure being the convex combination of two probability measures). Since the function $(\partial_{\mu}\mathcal{K})(y,\mu)$ is only uniquely determined up to a constant (since $\int_{\mathbb{R}^n}(\mu'(dy)-\mu(dy))=0$ being $\mu$ and $\mu'$ two probability measures),  we can choose the normalization condition given by 
$$\int_{\mathbb{R}^n}({\partial_{\mu}\mathcal{K})(y,\mu)\mu(dy)}=0.$$
If a function $\bar{\mathcal{K}} :\mathbb{R}^n \times \mathcal{P}(\mathbb{R}^n) \rightarrow \mathbb{R}$ depends also on $x\in \mathbb{R}^n$ we say that $\bar{\mathcal{K}}$ is G\^{a}teaux differentiable if $\bar{\mathcal{K}}(x,\cdot)$ is G\^{a}teaux differentiable for any $x \in \mathbb{R}^n$. In this case we write 
$$
\lim_{\epsilon \rightarrow 0^+}\frac{\bar{\mathcal{K}}(x,\mu+\epsilon(\mu'-\mu))-\bar{\mathcal{K}}(x,\mu)}{\epsilon}=\int_{\mathbb{R}^n}{\partial_{\mu}\bar{\mathcal{K}}(x,y,\mu)(\mu'(dy)-\mu(dy))}.
$$

 We formulate the following hypotheses on $\mathcal{V}$ (see \eqref{eq:J1}):
\textit{
\begin{itemize}
\item {\bf Hypotheses $\mathcal{V}$}:
\begin{enumerate}[i]
\item The map $\mathcal{V}$ is continuous from $\mathbb{R}^n \times \mathcal{P}(\mathbb{R}^n)$ to $\mathbb{R}$ 
(where $\mathcal{P}(\mathbb{R}^n)$ is equipped with the weak topology of convergence of measures). 
\item There is a positive function $V$ such that 
\begin{equation}\label{eq:conditionV}
|\partial^\alpha V(x)|\leq C_{\alpha} V(x) \quad \quad  \quad V(x)\leq C_1 V(y) \exp(C_2 |x-y|), \ x,y\in \mathbb{R}^n
\end{equation}
where $\alpha\in \mathbb{N}^n$ is a multiindex of length at most $|\alpha|\leq 2$, $C_{\alpha}$, $C_1$ and $C_2$ are positive constants, and growing to $+\infty$ as $|x|\rightarrow +\infty$. Furthermore there are three positive constants $c_1,c_2,c_3$, with $c_2>0$, such 
that   for any $\mu \in  \mathcal{P}(\mathbb{R}^n)$: 
\begin{equation}\label{eq:conditionV2}
 V (x)-c_1 \leq \mathcal{V}(x,\mu) \leq c_2 V(x) + c_3 \ x\in \mathbb{R}^n.
\end{equation}
\item  The map $\mathcal{V}$ is G\^{a}teaux differentiable and $\partial_{\mu}\mathcal{V}(x,y,\mu)$ is uniformly bounded from below and we have 
\begin{equation}\label{eq:conditionV3}
 \partial_{\mu}\mathcal{V}(x,y,\mu) \leq D_1+D_2 V(x) V(y), \ x,y\in \mathbb{R}^n
\end{equation} 
for some $D_1,D_2 \geq 0$. Furthermore whenever 
$$\partial_{\mu}\tilde{\mathcal{V}}(y,\mu)=\mathcal{V}(y,\mu)+\int_{\mathbb{R}^n}{\partial_{\mu}\mathcal{V}(x,y,\mu)\mu(dx)}$$ 
is well defined (namely when $\int_{\mathbb{R}^n}{V(x)\mu(dx)}<+\infty$), we require that $ \partial_{\mu}\tilde{\mathcal{V}}(\cdot,\mu)$ is a $C^{\frac{n}{2}+\delta}(\mathbb{R}^n,\mathbb{R})$ H\"older function for some $\delta>0$. 
\end{enumerate}
\item {\bf Hypothesis C$\mathcal{V}$}: the functional $\tilde{\mathcal{V}}$ is convex.
\item {\bf Hypothesis Q$V$}: the function $V$, in Hypotheses $\mathcal{V}$, is radially symmetric $V(x)=\bar{V}(|x|)$, where $\bar{V}$ is a $C^1(\mathbb{R}_+,\mathbb{R})$ increasing function for which there are constants $e_1,\epsilon>0,e_2,e_3\geq 0$ such that:
\begin{enumerate}[i]
\item $\bar{V}(r)\geq e_1 r^{2+\epsilon} -e_2$, 
\item $\bar{V}'(r) \leq e_3 (\bar{V}(r))^{\frac{3}{2}}$, $r=|x|$. 
\end{enumerate}
\end{itemize}
}

\begin{remark}
The conditions \eqref{eq:conditionV} are some standard requests on the weight function $V$ for having good properties in the Sobolev and Besov spaces on $\mathbb{R}^n$ with weight $V$ (see, i.e., \cite{Triebel1987,Schott1,Schott2}). We use some of these properties in an essential way in Lemma \ref{lemma_compact} below.
\end{remark}

\begin{remark}\label{remark_uniform}
An important consequence of Hypothesis $\mathcal{V}$\emph{i} is that if $\mu_k$ is a sequence in $\mathcal{P}(\mathbb{R}^n)$ converging weakly to $\mu$, as $k \rightarrow +\infty$, then for any  compact set  $K \subset \mathbb{R}^n$ we have $\sup_{x\in K}|\mathcal{V}(x,\mu)-\mathcal{V}(x,\mu_k)| \rightarrow 0 $. This fact is a consequence of the Prokhorov theorem (which says that $\mathcal{P}(\mathbb{R}^n)$ is a complete metric space) and  of  the Heine-Cantor theorem (which says that a continuous function from a compact metric space to a metric space is uniformly continuous). 
\end{remark}

\begin{remark}
Hypothesis C$\mathcal{V}$ is essentially used in two points of the present paper: in Theorem \ref{theorem_minimizer1}, where it is exploited for proving the uniqueness of the minimizer $\rho_0$, and in Theorem \ref{theorem_VE}, where the uniqueness proved in Theorem \ref{theorem_minimizer1} is applied to prove that potentials of the form \eqref{eq:form1} (below) satisfy the value functions convergence condition \eqref{LimE1}. In both cases Hypothesis C$\mathcal{V}$ guarantees uniqueness of the minimizer in the limit problem. If we do not assume Hypothesis C$\mathcal{V}$ we have to consider \emph{relaxed controls} (see \cite{Borkar} for the Markovian ergodic case and \cite{Lacker1} for controlled McKean-Vlasov dynamics).\\
It is important also to note that a monotonicity condition is required in the mean-field games literature in order to have uniqueness of Nash equilibrium (see, e.g. \cite{Carmonabook1,Lions2019}). More precisely if $\tilde{\mathcal{V}}$ is convex then $\partial_{\mu}\tilde{\mathcal{V}}$ is monotone, i.e.
$$\int_{\mathbb{R}^{n}}{[\partial_{\mu}\tilde{\mathcal{V}}(y,\mu)-\partial_{\mu}\tilde{\mathcal{V}}(y,\mu')](\mu(dy)-\mu'(dy))}\geq 0,$$
for any probability measures $\mu,\mu'\in \mathcal{P}(\mathbb{R}^n)$. 
\end{remark}

We consider the ergodic control problem given by the cost functions \eqref{eq:J1}. 
Since  the cost functional $J(\alpha,x_0)$ (on the left hand side of \eqref{eq:J1}) depends on the law of the controlled diffusion $X_t$ of the time averaged ergodic control problem,  it is legitimate to look at it as a McKean-Vlasov control problem.\\

We define 
\begin{equation}\label{eq:mathfrakJ}
\mathfrak{J}:=\text{ess sup}_{x_0 \in \mathbb{R}^n}\left(\inf_{\alpha \in C^1(\mathbb{R}^n,\mathbb{R}^n)}J(\alpha,x_0)\right)
\end{equation}
where $\text{ess sup}$ is the essential supremum over $x_0 \in \mathbb{R}^n$. In the ergodic case the optimal value $\mathfrak{J}$ is the analogous of the value function of the finite time optimal control problem. With an abuse of name we call $\mathfrak{J}$ the value function associated with the problem \eqref{eq:SDEmain} and the cost functional \eqref{eq:J1}.

\begin{remark}
There are two important observations to do about the initial conditions chosen in the definition of value function \eqref{eq:mathfrakJ}. The first one is that the function $x_0 \longmapsto \inf_{\alpha \in C^1(\mathbb{R}^n,\mathbb{R}^n)}J(\alpha,x_0)$ is almost surely constant in $x_0$ with respect to the Lebesgue measure (see Theorem \ref{theorem_optimalcontrol1}). This means that the $\text{ess sup}_{x_0 \in \mathbb{R}^n}$ is used only to exclude a set of measure zero with respect to $x_0$.\\
The second observation is that, although in Section \ref{section_existence} we consider only deterministic initial conditions, it is possible to extend, in a straightforward way, our analysis by considering
$$
\bar{J}(\alpha,p):=\limsup_{T \rightarrow +\infty}\frac{1}{T} \left(\int_0^T\mathbb{E}_{X_0 \sim p(x)dx}\left[\frac{|\alpha(X_t)|^2}{2}+\mathcal{V}(X_t,\law(X_t)) \right]dt \right),$$ 
where the process $X_t$ has an initial probability law, $\law(X_0)$, which is absolutely continuous with respect to Lebesgue measure of the form  $p(x)dx$, with $p$ a positive Lebesgue integrable function on $\mathbb{R}^n$, and such that $\int_{\mathbb{R}^n}{V(x)p(x)dx}<+\infty$. Indeed in both Theorem \ref{theorem_preliminary} and Lemma \ref{lemma_ergodic1} (below) we can replace the deterministic initial condition with a random one, of the previous type, obtaining the corresponding statement. This fact proves that 
$$\mathfrak{J}=\inf_{\alpha \in C^1(\mathbb{R}^n,\mathbb{R}^n)}\bar{J}(\alpha,p),$$
for any $p \in L^1(\mathbb{R}^n)$, where $\mathfrak{J}$ is the same constant as in definition \eqref{eq:mathfrakJ}. In this paper, we decided to treat in detail only the case of deterministic initial conditions in order to simplify the treatment of the general problem.
\end{remark}

\subsection{A family of potentials satisfying Hypotheses $\mathcal{V}$, C$\mathcal{V}$ and Q$V$}

In this section we discuss a class of functionals $\mathcal{V}$ satisfying Hypotheses $\mathcal{V}$ and C$\mathcal{V}$. More precisely  we consider the functionals $\mathcal{V}$ having the following form 
\begin{equation}\label{eq:form1}
\mathcal{V}(x,\mu)=V_0(x)+\int_{\mathbb{R}^n}{v_0(y)\mu(dy)}+\int_{\mathbb{R}^n}{v_1(x-y)\mu(dy)},
\end{equation}
where $V_0,v_0,v_1 \in C^{\frac{n}{2}+\epsilon}(\mathbb{R}^{n}), \epsilon >0$ and $\mu \in \mathcal{M}_c(\mathbb{R}^n)$ (where $ \mathcal{M}_c(\mathbb{R}^n)  $ is the space of signed measures on $\mathbb{R}^n$ having total mass less than $c\in \mathbb{R}_+$) . Furthermore we require that $V_0$ grows to plus infinity as $|x|\rightarrow +\infty$, and there is a function $V$, satisfying the relation \eqref{eq:conditionV} and Hypothesis Q$V$, such that $V_0(x) \sim V(x)$ as $|x| \rightarrow +\infty$ (where $\sim$ stands for $V_0(x)$ is bounded from above and below by positive constants times $V(x)$ as $|x| \rightarrow +\infty$). We also assume that $v_0,v_1$ are bounded, $v_1(x)=v_1(-x)$ and that there exists a positive measure $\pi$ on $\mathbb{R}^n$ such that, for any $x \in \mathbb{R}^{n}$, $v_1(x)=\int_{\mathbb{R}^n}{e^{-ikx}\pi(dk)}$ (i.e. $v_1$ is the Fourier transform of a positive measure). 

\begin{theorem}\label{theorem_bochner}
The functional $\mathcal{V}$ of the form \eqref{eq:form1} under the above assumptions on $V_0,v_0,v_1$ satisfies Hypotheses $\mathcal{V}$ and C$\mathcal{V}$. 
\end{theorem}
\begin{proof}
Hypothesis $\mathcal{V}$\emph{i} follows from the fact that $\mathcal{V}(x,\cdot)$ is a sum of affine bounded functionals on $\mathcal{M}_c(\mathbb{R}^n)$.
Since $v_0,v_1$ are bounded and $V_0$ grows at $+\infty$ when $|x|\rightarrow + \infty$, $\mathcal{V}$ satisfies Hypothesis $\mathcal{V}$\emph{ii}.
By an explicit computation we have
$$\partial_{\mu}(\mathcal{V})(x,y,\mu)=v_0(y)+v_1(x-y)$$
hence, since, for the previous assumptions, $v_0$ and $v_1$ are bounded and regular enough,  $\mathcal{V}$ satisfies $\mathcal{V}$\emph{iii}.\\
Furthermore we get, by the definition \eqref{functionalofmu} and the fact that the integral of $v_0$ in \eqref{eq:form1} is constant
$$\tilde{\mathcal{V}}(\mu)=\int_{\mathbb{R}^n}{(V_0(x)+v_0(x))\mu(dx)}+\int_{\mathbb{R}^{2n}}{v_1(x-y)\mu(dx)\mu(dy)}$$
and so 
$$\partial_{\mu}^2(\tilde{\mathcal{V}})(x,y,\mu)= 2v_1(x-y).$$
The previous equation implies that, if the bilinear form 
$$B(\tilde{\mu},\tilde{\nu})=2\int_{\mathbb{R}^n}v_1(x-y)\tilde{\nu}(dx)\tilde{\mu}(dy), \quad \tilde{\mu},\tilde{\nu}\in \mathcal{M}_{2c}(\mathbb{R}^n)$$
is nonnegative definite, then $\tilde{\mathcal{V}}$ is convex having nonnegative definite second differential. Since $v_1$ is continuous and bounded, by Theorem XVIII of Chapter VII in \cite{Schwartz1966} (see also \cite[Theorem IX.10]{ReedSimon2}), the form $B$ is positive definite if and only if $v_1$ is a positive definite function. By Bochner's theorem (see, e.g., \cite[Theorem IX.9]{ReedSimon2}), $v_1$ is a positive definite continuous function if and only if it is the Fourier transform of a positive measure. This complete the proof of the theorem.
\end{proof}

\begin{remark}
Using Theorem \ref{theorem_bochner} it is possible to build other functionals satisfying Hypotheses $\mathcal{V}$, C$\mathcal{V}$ and Q$V$. Indeed we can, e.g., compose functionals of the form \eqref{eq:form1} with the derivatives of an  homogeneous symmetric polynomial $P:\mathbb{R}^k \rightarrow \mathbb{R}$ which is convex and it has positive partial derivatives on $\mathbb{R}_+^k$. More precisely let $v_1, ..., v_k$ be positive functions satisfying the same conditions of $v_1$ in Theorem \ref{theorem_bochner}, and consider 
\begin{multline}\label{eq:furtherexample}
\mathcal{V}(x,\mu)=V_0(x)+\\
+\sum_{\ell=1}^k \partial_{z_{\ell}}(P_k)\left(\int_{\mathbb{R}^{2n}}{v_1(y_1'-y_1)\mu(dy_1')\mu(dy_1)},\cdots,\int{v_k(y_k'-y_k)\mu(dy_k')\mu(dy_k)} \right) \int_{\mathbb{R}^n}{v_{\ell}(x-y_{\ell})\mu(dy_{\ell})}.
\end{multline}
Using the fact that $P_k$ is an homogeneous polynomial, we get that the functional $\tilde{\mathcal{V}}$ associated with the operator $\mathcal{V}$ given in \eqref{eq:furtherexample}, we get 
\[\tilde{\mathcal{V}}(\mu)=\int_{\mathbb{R}^k}V_0(y)\mu(dy)+\operatorname{deg}(P_k) P_k \left(\int_{\mathbb{R}^{2n}}{v_1(y_1'-y_1)\mu(dy_1')\mu(dy_1)},\cdots,\int{v_k(y_k'-y_k)\mu(dy_k')\mu(dy_k)} \right).\]
This implies that $\tilde{\mathcal{V}}$ is convex since it is the sum of a linear function, and the composition of positive convex functionals $\int_{\mathbb{R}^{2n}}{v_{\ell}(y_{\ell}-y'_{\ell}) \mu(dy_{\ell})\mu(dy_{\ell}')}$ and a convex function with positive derivatives $P_k:\mathbb{R}_+^k \rightarrow \mathbb{R}$. Since the linear combination (with positive coefficients) of convex functional is convex we can use the previous construction to build general non quadratic, and also non polynomial, cost functionals.
\end{remark}

\section{The McKean-Vlasov optimal control problem}\label{section_existence}

\subsection{The ergodic control problem}

We are searching for the control function $\alpha \in C^1$ which minimizes the functional \eqref{eq:J1}. First of all we need some results and notations concerning equations of the form \eqref{eq:SDEmain} when $\alpha \in C^1$ is admitting an invariant measure. We denote by $\mu_{t,x_0}$ the probability measure on $\mathbb{R}^n$ giving the distribution of $X_t$ when $X_0=x_0$, $x_0\in \mathbb{R}^n$. We also write $\tilde{\mu}_{t,x_0}$ for the following time averaged measure 
$$\tilde{\mu}_{t,x_0}(x)= \frac{1}{t} \int_0^t{\mu_{\tau,x_0}(dx)d\tau}, x \in \mathbb{R}^n  \text{ and }t\in \mathbb{R}_+$$
We denote by $T_t$, $t \in \mathbb{R}_+$, the (sub)Markovian semigroup associated with SDE \eqref{eq:SDEmain}, namely if $f\in L^1(\mathbb{R}^n,\mu_{t,x})\equiv L^1(\mu_{t,x})$  we  have
\begin{equation}\label{eq:invariant}
T_t(f)(x)=\int_{\mathbb{R}^n}f(y)\mu_{t,x}(dy)=\mathbb{E}_{x}[f(X_t)].
\end{equation}
We denote by $L_{\alpha}:C^{\infty}(\mathbb{R}^n) \rightarrow C^1(\mathbb{R}^n)$ the operator 
\begin{equation}\label{eq:generator}
L_{\alpha}(g)(x)=\Delta g(x)+\sum_{i=1}^n \alpha^i(x) \partial_{x_i}(g)(x), \quad g\in C^{\infty}(\mathbb{R}^n) 
\end{equation}
 
We say that a probability measure $\mu$ is an invariant measure for the process $X_t$, $t \geq 0$, or equivalently, for the semigroup $T_t$ if, for any bounded measurable function $f$, we have
$$\int_{\mathbb{R}^n}T_t(f)(x)\mu(dx)=\int_{\mathbb{R}^n}f(x)\mu(dx).$$
We say that a probability measure $\mu$ is infinitesimal invariant for $T_t$, $t\geq 0$, and we write $L_{\alpha}^*(\mu)=0$, if for any $g\in C^{\infty}(\mathbb{R}^n)$ with compact support we have 
\begin{equation}\label{eq:invariantinfinitesimal}
\int_{\mathbb{R}^n}L_{\alpha}(g)(x)\mu(dx)=0.
\end{equation}
 
\begin{proposition}\label{theorem_preliminary}
Consider an SDE of the form \eqref{eq:SDEmain} with $\alpha \in C^1$, and suppose that it admits an invariant measure $\mu$. Then the following assertions hold:
\begin{enumerate}[i]
\item $T_t$ is strong Feller,
\item $\mu$ is the unique ergodic invariant measure of $T_t$,
\item $\mu$ is absolutely continuous with respect to the Lebesgue measure,
\item for any $x_0 \in \mathbb{R}^n$, $\tilde{\mu}_{t,x_0} \rightarrow \mu$ weakly as $t \rightarrow +\infty$,
\item  for any $x_0 \in \mathbb{R}^n$, $\mu_{t,x_0} \rightarrow \mu$ weakly as $t \rightarrow +\infty$,
\item if further $|\alpha|^2 \in L^1(\mu)$ then for any $f \in L^1(\mu)$ we have $\lim_{t \rightarrow +\infty}\frac{1}{t}\int_0^t{T_s(f)(x_0)}=\int_{\mathbb{R}^n}{f(x)\mu(dx)}$ for $\mu$-almost all $x_0 \in \mathbb{R}^n$,
\item finally $\mu$ is the unique invariant measure of $T_t$ if and only if it is the unique solution to the equation $L_{\alpha}^*(\mu)=0$. 
\end{enumerate}
\end{proposition}
\begin{proof}
By \cite[Proposition 2.2.12]{Lorenzi} $T_t$ is irreducible and strong Feller. This implies that $X_t$ has an unique ergodic invariant measure, from Doob's Theorem in \cite[Theorem 4.2.1]{DaPrato}, which means that $\mu$ is the unique solution to the Fokker-Planck equation $L^*_{\alpha}(\mu)=0$, where $L_{\alpha}$ is the infinitesimal generator of $T_t$ (which is the unique extension of the operator \eqref{eq:generator}) and $L^*_{\alpha}$ its adjoint, (proving the point \emph{ii}). Furthermore, since $\alpha$ is $C^1$ and $L_{\alpha}$ is uniformly elliptic, by \cite[Corollary 1.5.3]{Rockner}, we have that $\mu$ is absolutely continuous with respect to Lebesgue measure. 
Points \emph{iv} and \emph{v} are consequences of \cite[Theorem 4.2.1]{DaPrato}. Furthermore,  using the fact that $|\alpha|^2 \in L^1(\mu)$ and by Theorem 5.2.9 of \cite{Rockner}, the semigroup $T_t$  is a strongly continuous semigroup on $L^1(\mu)$. By Remark 1 in \cite[Chapter XII Section 1]{Yosida}, this implies point \emph{vi}. The point \emph{vii} follows by the uniqueness of the invariant measure. 
\end{proof}

\begin{remark}
By a classical result a sufficient condition for the existence of an invariant measure is that $\alpha$ is of the form $\alpha= -DU-G$, with $U\in C_{1}^{1+\beta}(\mathbb{R}^N)$ for some $\beta\in (0,1)$, $G\in C^1(\mathbb{R}^N,(\mathbb{R}^N)$, $e^{-U}\in  L^1(\mathbb{R}^N),|G|e^{-U}\in L^1(\mathbb{R}^N)$, $\operatorname{div} G=<G,DU>$. In this case $\mu(dy)=\frac{\exp(-U(x)dy)
}{\int \exp{-U(x)}dx}$ is symmetric in $L^2(\mu)$ (see, e.g., \cite{Lorenzi}, Chapter 8 Theorem 8.1.26).
\end{remark}

\begin{remark}
It is important to note that the condition $\alpha \in C^1(\mathbb{R}^n,\mathbb{R}^n)$ is an essential hypothesis in Proposition \ref{theorem_preliminary} (and consequently in Lemma \ref{lemma_ergodic1} and Theorem \ref{theorem_optimalcontrol1} below). Indeed, when $\alpha$ is in general only measurable and $L^2$ with respect (some) invariant measure $\mu$ there is the possibility of multiple invariant measures even in one dimension (see, e.g., the discussions in \cite[Chapter 4]{Rockner}).
\end{remark}

\begin{remark}
Since, when $\alpha \in C^1(\mathbb{R}^n,\mathbb{R}^n)$, and by Proposition \ref{theorem_preliminary} \emph{vii}, the two notions of invariant measures, namely then ones given by equations \eqref{eq:invariant} and \eqref{eq:invariantinfinitesimal} respectively, are equivalent, hereafter we call invariant measure a measure which respect both equations \eqref{eq:invariant} and \eqref{eq:invariantinfinitesimal}.
\end{remark}

In the next lemma we shall provide  a sufficient condition for the existence of an invariant measure for the SDE \eqref{eq:SDEmain} admitting a probability density. This allows to obtain a cost functional expressed in terms of a probability density notably simplifying our minimization problem. Hereafter we use the following abuse of notation: If $\mu$ is a measure on $\mathbb{R}^n$ absolutely continuous with respect to the Lebesgue measure with density $\rho$ (namely $d\mu(x)=\rho(x)dx$) we write $\mathcal{V}(x,\rho)$ instead of the more precise $\mathcal{V}(x,\mu)=\mathcal{V}(x,\rho dx)$.

\begin{lemma}\label{lemma_ergodic1}
Under hypotheses $\mathcal{V}$i and $\mathcal{V}$ii, if $J(\alpha,x_0)$  as given by \eqref{eq:J1} (with $\alpha \in C^1$) is not equal to $+\infty$  there exists an unique and ergodic invariant probability density measure $\rho_{\alpha}\in W^{1,\frac{n}{2}}(\mathbb{R}^n)$ for the SDE \eqref{eq:SDEmain} so that $\mu_{\alpha}(dx)=\rho_{\alpha}(x)dx,$ with $\mu_{\alpha}$ the invariant ergodic probability measure for the SDE \eqref{eq:SDEmain}. Furthermore we have 
$$\tilde{J}(\alpha,\rho_{\alpha}) \leq J(\alpha,x_0) $$
for almost all $x_0 \in \mathbb{R}^n$ with respect to Lebesgue measure, where 
\begin{equation}\label{eq:tildeJ}
\tilde{J}(\alpha,\rho_{\alpha}):=\int_{\mathbb{R}^n}{\left(\frac{|\alpha(x)|^2}{2}+\mathcal{V}(x,\rho_{\alpha}) \right) \rho_{\alpha}(x) dx.}
\end{equation}

\end{lemma}
\begin{proof}
Under hypothesis $\mathcal{V}$\emph{ii} when $J(\alpha,x_0)$ is finite, for any $x_0 \in \mathbb{R}^n$, $\tilde{\mu}_{t,x_0}$, indexed by $t \in \mathbb{R}_+$, is a family of tight measures. Indeed we have that, for any $t>0$: 
\begin{align*}
\int_{\mathbb{R}^n}{ V(x)\tilde{\mu}_{t,x_0}(dx)}=&\frac{1}{t}\int_0^t \int_{\mathbb{R}^n}{  V(x)\mu_{\tau,x_0}(dx)d\tau} \\
\leq&\frac{1}{t}\int_{\mathbb{R}^n}{\int_0^t  \mathcal{V}(x,\mu_{\tau,x_0})\mu_{\tau,x_0}(dx)d\tau} + c_1 \\
\leq&\frac{1}{t} \left(\int_0^t\mathbb{E}_{x_0}\left[\frac{|\alpha(X_{\tau})|^2}{2}+\mathcal{V}(X_{\tau},\law(X_{\tau})) \right]d\tau \right)+c_1< C
\end{align*}
for some $C \in \mathbb{R}$, where in the last step we used that $J(\alpha,x_0)<+\infty$. \\
Since $V$ is a function growing to infinity as $|x|\rightarrow + \infty$, the family  $(\tilde{\mu}_{t,x_0}(dx), t>0)$ is necessarily tight. Now let $\mu(dx)$ be any weak limit of a subsequence of $\tilde{\mu}_{t,x_0}(dx)$, as $t \rightarrow +\infty$, then $\mu(dx)$ is an invariant probability measure for equation \eqref{eq:SDEmain}. Indeed let $f$ be a $C^{\infty}$ function with compact support then, by It\^o formula, for $t>0$:
$$\int_{\mathbb{R}^n}{L_{\alpha}(f)(x)\tilde{\mu}_{t,x_0}(dx)}=\frac{1}{t}(\mathbb{E}[f(X_t)]-f(x_0)).$$
Since $f$ has compact support we have that $\lim_{t\rightarrow +\infty}\frac{1}{t}(\mathbb{E}[f(X_t)]-f(x_0))=0$, which implies,  $\alpha$  being locally bounded, that 
$$0=\lim_{t\rightarrow +\infty}\int_{\mathbb{R}^n}{L_{\alpha}(f)(x)\tilde{\mu}_{t,x_0}(dx)}=\int_{\mathbb{R}^n}{L_{\alpha}(f)(x)\mu(dx)}.$$
This means that $L^*_{\alpha}(\mu)=0$ and thus $\mu$ is an invariant probability measure for equation \eqref{eq:SDEmain}.
By Proposition \ref{theorem_preliminary} \emph{iii}, there exists a positive $L^1(\mathbb{R}^n)$ function $\rho_{\alpha}(x)$ such that $\mu(dx)=\rho_{\alpha}(x)dx$.\\
What remains to be proved is that $J(\alpha,x_0) \geq \tilde{J}(\alpha,\rho_{\alpha})$ (Lebesgue-)almost surely with respect to  $ x_0 \in \mathbb{R}^n$.\\
We have that $ \liminf_{t \rightarrow +\infty} \int_{\mathbb{R}^n}{\frac{|\alpha(x)|^2}{2} \tilde{\mu}_{t,x_0}(dx)} \geq \int_{\mathbb{R}^n}{\frac{|\alpha(x)|^2}{2} \rho_{\alpha}(x) dx}$. Indeed, for any $N\in \mathbb{N}$: 
\begin{align*}
\int_{\mathbb{R}^n}{\frac{|\alpha(x)|^2 \wedge N}{2} \rho_{\alpha}(x) dx} =& \lim_{t \rightarrow +\infty} \int_{\mathbb{R}^n}{\frac{|\alpha(x)|^2 \wedge N}{2} \tilde{\mu}_{t,x_0}(dx)}\\
\leq & \liminf_{t \rightarrow +\infty} \int_{\mathbb{R}^n}{\frac{|\alpha(x)|^2}{2} \tilde{\mu}_{t,x_0}(dx)}<+\infty.
\end{align*}
Since $\lim_{N\rightarrow +\infty}\int_{\mathbb{R}^n}{\frac{|\alpha(x)|^2 \wedge N}{2} \rho_{\alpha}(x)dx}=\int_{\mathbb{R}^n}{\frac{|\alpha(x)|^2}{2} \rho_{\alpha}(x) dx}$ the stated inequality is proved.\\
Now we want to prove that 
\begin{equation}\label{eq:limitrho}
\limsup_{t\rightarrow +\infty}\int_{\mathbb{R}^n}{\mathcal{V}(x,\mu_{t,x_0})\mu_{t,x_0}(dx)}=\int_{\mathbb{R}^n}{\mathcal{V}(x,\mu)\rho_{\alpha}(x)dx},
\end{equation}
almost surely with respect to $x_0 \in \mathbb{R}^n$. Let $t_m \rightarrow +\infty$ be a sequence in $\mathbb{R}_+$ which realizes the $\limsup$ in \eqref{eq:limitrho}. By Proposition \ref{theorem_preliminary} \emph{vi} and denoting by $B_K$ the ball of radius $K \in \mathbb{N}$ and center in $0$ we have that 
\begin{align*}
\lim_{m \rightarrow +\infty}\int_{\mathbb{R}^n \setminus B_K}{V(x)\tilde{\mu}_{t_{m},x_0}(dx)}=&\lim_{m \rightarrow +\infty}\frac{1}{t_{m}}T_{t_m}(\mathbb{I}_{\mathbb{R}^n \setminus B_K} V)(x_0)
\\=&\int_{\mathbb{R}^n \setminus B_K}{V(x)\rho_{\alpha}(x)dx},
\end{align*}
for almost all $x_0 \in \mathbb{R}^n$ and for all $K \in \mathbb{N}$. Since the positive measure $V(x)\rho_{\alpha}(x)dx$ is regular, this means that the sequence of  positive measures $V(x)\tilde{\mu}_{t_m,x_0}(dx)$ is tight. By Hypothesis $\mathcal{V}$\emph{ii}, the tightness of $V(x)\tilde{\mu}_{t_m,x_0}(dx)$ implies the tightness of the sequence of signed measures (with total mass uniformly bounded) $\frac{1}{t_m}\int_0^{t_m}{\mathcal{V}(x,\mu_{s,x_0})\mu_{s,x_0}(dx)ds}$. On the other hand, by Remark \ref{remark_uniform} (in Section \ref{section_technical}) and using the fact that, by Proposition \ref{theorem_preliminary} \emph{iv}, $\mu_{t,x_0}\rightarrow \mu$ weakly, as $t \rightarrow +\infty$, we have
\begin{multline}
\lim_{m \rightarrow +\infty}\frac{1}{t_m}\int_0^{t_m}\int_{\mathfrak{K}}{\mathcal{V}(x,\mu_{s,x_0})\mu_{s,x_0}(dx)ds}=\\
=\lim_{m \rightarrow +\infty}\left(\int_{\mathfrak{K}}{\mathcal{V}(x,\mu)\tilde{\mu}_{s,x_0}(dx)ds} 
 +\frac{1}{t_m}\int_0^{t_m}\int_{\mathfrak{K}}{(\mathcal{V}(x,\mu_{s,x_0})-\mathcal{V}(x,\mu))\mu_{s,x_0}(dx)ds}\right)\\
=\int_{\mathfrak{K}}\mathcal{V}(x,\mu)\rho_{\alpha}(x)dx \label{eq:Vcompact}
\end{multline}
for any compact set $\mathfrak{K} \subset \mathbb{R}^n$. Since $\frac{1}{t_m}\int_0^{t_m}{\mathcal{V}(x,\mu_{s,x_0})\mu_{s,x_0}(dx)ds}$ has a uniformly bounded mass and  is tight, relation \eqref{eq:Vcompact} implies that $\frac{1}{t_m}\int_0^{t_m}{\mathcal{V}(x,\mu_{s,x_0})\mu_{s,x_0}(dx)ds}$ converges as $m \rightarrow \infty$ to $\mathcal{V}(x,\mu)\rho_{\alpha}(x)(dx)$, weakly for (Lebesgue) almost all $x_0 \in \mathbb{R}^n$. This proves equality \eqref{eq:limitrho} and concludes the proof.
\end{proof}


\subsection{A lower bound for the functional \eqref{eq:tildeJ}}

In order to minimize the functional \eqref{eq:tildeJ} with respect to $\rho$, for $\rho \in W^{1,\frac{n}{2}}(\mathbb{R}^n)$, $\rho(x) \geq 0$ and $\int_{\mathbb{R}^n}{\rho(x)dx}=1$, we set 
\begin{equation}\label{eq:Crho}
\mathcal{C}_{\rho}=\{\alpha \in  C^1(\mathbb{R}^n,\mathbb{R}^n),  \  L^*_{\alpha}(\rho)=0 \text{ and } |\tilde{J}(\alpha,\rho)|<+\infty\}.
\end{equation}
Then  $\mathcal{C}_{\rho}$ is the subset of  $C^1(\mathbb{R}^n,\mathbb{R}^n)$ vector fields $\alpha_{\rho}\in \mathcal{C}_{\rho}$ such that  $L^*_{\alpha_{\rho}}(\rho)=0$ (where $L^*_{\alpha_{\rho}}$ is the adjoint of the infinitesimal generator of the equation \eqref{eq:SDEmain} and the previous equality is understood in a distributional sense) and $|\tilde{J}(\alpha_{\rho},\rho)|<+\infty$.\\

\begin{remark}\label{remark_rhoalpha}
Suppose that $\alpha \in C^1(\mathbb{R}^n,\mathbb{R}^n)$ such that $J(\alpha,x_0) <+\infty$, then by Lemma \ref{lemma_ergodic1} there is a unique positive probability density $\rho_{\alpha}$ which is invariant and thus, since $\alpha \in C^1(\mathbb{R}^n,\mathbb{R}^n)$ by Proposition \ref{theorem_preliminary},  it satisfies the equation $L_{\alpha}^*(\rho_{\alpha})=0$. This implies that $\alpha \in \mathcal{C}_{\rho_{\alpha}}$, where $
\mathcal{C}_{\rho_{\alpha}}$ is defined by equation \eqref{eq:Crho} with $\rho=\rho_{\alpha}$.
\end{remark}

We introduce the following energy functional,  for $\rho \in W^{1,\frac{n}{2}}(\mathbb{R}^n)$,
\begin{equation}\label{eq:energy1}
\mathcal{E}(\rho):=\mathcal{E}_K(\rho)+\mathcal{E}_{P}(\rho)=\int_{\mathbb{R}^n}\frac{|\nabla \rho|^2}{\rho}dx+ \int_{\mathbb{R}^n}{\mathcal{V}(x,\rho)\rho(x)dx}. 
\end{equation}
where the two terms on the right hand side correspond to the kinetic $\mathcal{E}_K(\rho)$  and potential $\mathcal{E}_P(\rho)  $ energies, respectively.\\
The next lemma states a useful monotonicity property of the cost functional $\tilde{J}$.

\begin{lemma}\label{lemma_ergodic2}
For any given $\rho \in W^{1,\frac{n}{2}}(\mathbb{R}^n)$ we have 
$$\mathcal{E}(\rho)=\tilde{J}\left(\frac{\nabla \rho}{\rho},\rho\right) \leq \inf_{\alpha \in \mathcal{C}_{\rho}} \tilde{J}(\alpha,\rho),$$
where $\tilde{J}(\alpha,\rho)$ is defined in \eqref{eq:tildeJ}.
\end{lemma}
\begin{proof}
By \cite[Chapter 3, Theorem 3.1.2]{Rockner}, if $\rho$ is the density of the invariant measure of the SDE \eqref{eq:SDEmain} we have that 
$$\int_{\mathbb{R}^n}{\frac{|\nabla\rho(x)|^2}{\rho^2(x)}\rho(x)dx}\leq\int_{\mathbb{R}^n}{|\alpha(x)|^2\rho(x)dx},$$
for any $\alpha \in \mathcal{C}_{\rho}$, with the equality holding if and only if $\alpha=\frac{\nabla \rho}{\rho}$. Since $\int_{\mathbb{R}^n}{\mathcal{V}(x,\mu)\rho(x)dx}$ does not depend on $\alpha$ but only on the invariant measure $\rho(x) dx$, the theorem is proved.
\end{proof}

\subsection{Minimizer of the energy functional} 

We want to minimize the function $\mathcal{E}(\rho)$ given by \eqref{eq:energy1} under the condition $\int_{\mathbb{R}^n}{\rho(x)dx}=1$. It is useful to introduce the following variable $\phi=\sqrt{\rho}$. With this notation the energy functional \eqref{eq:energy1} becomes
\begin{equation}\label{eq:energy21} \mathcal{E}(\phi^2)=\int_{\mathbb{R}^n}{\left(\frac{|\nabla \phi|^2}{2}+\mathcal{V}(x,\phi^2)\phi^2(x)\right)dx}, \end{equation}
with $\phi \in L^2(\mathbb{R}^n)$ satisfying the condition $\int_{\mathbb{R}^n}{\phi^2(x)dx}=1$.  \\
The following result states that the above energy functional admits a unique minimizer which is strictly positive.

\begin{lemma}\label{theorem_minimizer1}
Under hypotheses $\mathcal{V}$ and  C$\mathcal{V}$ the variational problem \eqref{eq:energy1}, with $\phi \in L^2(\mathbb{R}^n)$ satisfying the condition $\int_{\mathbb{R}^n}{\phi^2(x)dx}=1$, admits a unique minimizer $\rho_0=\phi_0^2$. Furthermore $\phi_0$ is $C^{2+\epsilon}(\mathbb{R}^n)$ for some $\epsilon >0$, it is strictly positive and  satisfies (weakly) the equation 
\begin{equation}\label{eq:variational1}
-\Delta \phi_0(x) + 2 \mathcal{V}(x,\phi_0^2)\phi_0(x) + 2 \int_{\mathbb{R}^n}{\partial_{\mu} \mathcal{V}(y,x,\phi_0^2)\phi_0^2(y)dy} \phi_0(x) = \mu_0 \phi_0(x),
\end{equation}  
where the uniquely determined constant $\mu_0$ given by
\begin{equation}\label{eq:mu0}
\mu_0=2\mathcal{E}(\phi_0^2)+ \int_{\mathbb{R}^n}\partial_{\mu} \mathcal{V}(y,x,\phi_0^2)\phi_0^2(y) \phi_0^2(x)dydx 
\end{equation} 
\end{lemma}
\begin{proof}
By Hypothesis C$\mathcal{V}$ the functional $\mathcal{E}_P(\rho)$ is convex and by the property of Fisher information (see Theorem \ref{theorem_fisher} below), $\mathcal{E}_K(\rho)$ is convex and strictly convex when it is finite. Furthermore by Hypothesis $\mathcal{V}$\emph{ii} $\mathcal{E}$ is coercive in $\phi_0$ (in the sense that $\mathcal{E}(\phi^2) \geq C ||\phi^2||_{H^1}$). This implies that there exists a unique minimizer $\phi_0=\sqrt{\rho_0}$.\\
On the other hand, making a variation of the form $\phi_0 +\epsilon \delta\phi$, where $\epsilon>0$ and $\delta \phi$ is supposed to be a smooth compactly supported function, under the additional constraint given by the normalization condition for $(\phi_0)^2$, by the regularity property given by Hypothesis $\mathcal{V}$\emph{iii}, the minimizer $\phi_0$ must satisfy (in a weak sense) equation \eqref{eq:variational1}. For determining the Lagrange multiplier $\mu_0$ it is sufficient to multiply both sides of equation \eqref{eq:variational1} by $\phi_0$ and then integrate by parts.\\
Using a bootstrap argument, beginning by $(\partial_{\mu}\tilde{\mathcal{V}})(\cdot,\phi_0^2) \in C^{\frac{n}{2}+\epsilon'}$, for some $\epsilon'>0$ by Hypothesis $\mathcal{V}$\emph{iii} and by elliptic regularization property of the Laplacian (see Theorem 8.10 in \cite{Gilbarg}), we obtain that $\phi_0 \in H^{\frac{n}{2}+\epsilon}_{loc}(\mathbb{R}^2)$ and thus $\phi_0 \in C^{\epsilon}(\mathbb{R}^n)$. Exploiting the regularity results for the Poisson equation (see Theorem 4.3 in \cite{Gilbarg}), we have that $\phi_0 \in C^{2+\epsilon}(\mathbb{R}^n)$.\\
Finally, equation \eqref{eq:variational1} implies that $\phi_0$ is the ground state of a quantum mechanical system on $\mathbb{R}^n$ with potential $2 \partial_{\mu}\tilde{\mathcal{V}}(x,\phi^2_0)$ (where $\tilde{\mathcal{V}}$ is defined in \eqref{functionalofmu}). Since, by Hypotheses $\mathcal{V}$\emph{ii} and $\mathcal{V}$\emph{iii}, $2\mathcal
{V}(x,\phi^2_0)+2 \partial_{\mu}\tilde{\mathcal{V}}(x,\phi^2_0)$ is bounded from below and diverges to infinity as $|x|\rightarrow +\infty$, by \cite[Theorem XIII.47]{ReedSimon} we have that $\phi_0$ is strictly positive.
\end{proof}

\begin{remark}\label{remark_minimizer1}
In Lemma \ref{theorem_minimizer1} Hypothesis C$\mathcal{V}$ is only used to prove the uniqueness of the minimizer $\rho_0$. Indeed in order to prove existence and positivity of $\phi_0$ we need only Hypotheses $\mathcal{V}$.
\end{remark}

\begin{remark}
The minimizer $\rho_0$ in Lemma \ref{theorem_minimizer1} satisfies the following equation 
\begin{equation}\label{eq:variational2}
-\Delta \rho_0(x)+\frac{|\nabla \rho_0(x) |^2}{2\rho_0(x)} +  \mathcal{V}(x,\rho_0)\rho_0(x) +  \int_{\mathbb{R}^n}{\partial_{\mu} \mathcal{V}(y,x,\rho_0)\rho_0(y)dy} \rho_0(x) = \mu_0 \rho_0(x),
\end{equation}
as easily deduced from \eqref{eq:variational1}.  
\end{remark}

\subsection{Existence and uniqueness of the optimal control}

Finally we obtain the explicit form of the optimal control:
\begin{theorem}\label{theorem_optimalcontrol1}
Under Hypotheses $\mathcal{V}$ and C$\mathcal{V}$, the logarithmic gradient  of the unique minimizer $\rho_0=\phi_0^2$ of $\mathcal{E}$, that is  $\alpha=\frac{\nabla \rho_0}{\rho_0}$, is the optimal control for the problem \eqref{eq:J1} for almost every $x_0 \in \mathbb{R}^n$ with respect to the Lebesgue measure.
\end{theorem}

In order to prove Theorem \ref{theorem_optimalcontrol1} we need the following lemma. 

\begin{lemma}\label{lemma_optimalE}
Under Hypotheses $\mathcal{V}$ and C$\mathcal{V}$ we have that 
\begin{equation}
J\left(\frac{\nabla \rho_0}{\rho_0},x_0\right)=\mathcal{E}(\rho_0),
\end{equation}
where $\rho_0=\phi_0^2$ is the unique minimizer of $\mathcal{E}$.
\end{lemma}
\begin{proof}
We have that $\mu(dx)=\rho_0(x)dx$ is the unique ergodic invariant probability measure of the strong Feller SDE \eqref{eq:SDEmain} with $\alpha=\frac{\nabla \rho_0}{\rho_0}$. By the definition of $\mathcal{E}$ and equation \eqref{eq:variational1} we have that $x \longmapsto |\alpha(x)|^2=\frac{|\nabla \rho_0(x)|^2}{\rho_0(x)^2}$ belongs to $ L^1(\mu)$. This implies, using  Proposition \ref{theorem_preliminary} \emph{vi}, that we have 
$$\lim_{t \rightarrow +\infty}\frac{1}{t}\int_0^t\mathbb{E}_{x_0}[|\alpha(X_s)|^2]ds= \lim_{t \rightarrow +\infty}\frac{1}{t}\int_0^t T_s(|\alpha(x)|^2)(x_0)ds= \int_{\mathbb{R}^n}|\alpha(x)|^2\rho_0(x)dx,$$
for (Lebesgue) almost every $x_0 \in \mathbb{R}^n$ (this is due to the fact that $\mu$ is absolutely continuous and $\rho_0$ is strictly positive). The proof of the fact that 
$$\limsup_{t\rightarrow +\infty}\frac{1}{t}\int_{0}^t{\mathbb{E}_{x_0}[\mathcal{V}(X_s,\law(X_s)]ds}=\int_{\mathbb{R}^n}{\mathcal{V}(x,\mu_0)\rho_0(x)dx},$$
is given in Lemma \ref{lemma_ergodic1} (see equation \eqref{eq:limitrho} and what follows from it).
\end{proof}

\begin{proof}[Proof of Theorem \ref{theorem_optimalcontrol1}]
By Lemma \ref{lemma_optimalE}, and the definition of $\mathfrak{J}$ (given in equation \eqref{eq:mathfrakJ}) we have that 
\begin{equation}\label{eq:inequalitymathfrakJ}
\mathfrak{J}\leq \text{ess sup}_{x_0 \in \mathbb{R}^n} J\left(\frac{\nabla \rho_0}{\rho_0},x_0\right)=\mathcal{E}(\rho_0). \end{equation}
In order to prove the statement of the theorem, it is sufficient to prove that
\[\mathcal{E}(\rho_0) \leq \mathfrak{J},\]
indeed, by Lemma \ref{lemma_optimalE} and inequality \eqref{eq:inequalitymathfrakJ}, this implies that $\mathfrak{J}\leq \text{ess sup}_{x_0 \in \mathbb{R}^n} J\left(\frac{\nabla \rho_0}{\rho_0},x_0\right)$ and thus the thesis.
By Lemma \ref{lemma_ergodic1}, we have $\tilde{J}(\alpha,\rho_{\alpha}) \leq J(\alpha,x_0)$ and by Lemma \ref{lemma_ergodic2}, and since, by Remark \ref{remark_rhoalpha}, $\alpha \in \mathcal{C}_{\rho_{\alpha}}$, we get, for any fixed $\alpha \in C^1(\mathbb{R}^n,\mathbb{R}^n)$ such that $ J(\alpha,x_0)<+\infty$, 
\[\mathcal{E}(\rho_{\alpha})=\tilde{J}\left(\frac{\nabla \rho_{\alpha}}{\rho_{\alpha}},\rho_{\alpha}\right) 
 \leq \inf_{\hat{\alpha} \in \mathcal{C}_{\rho_{\alpha}}} \tilde{J}(\hat{\alpha},\rho) \leq \tilde{J}(\alpha,\rho_{\alpha}). \]
 Combining the previous two inequalities and Lemma \ref{theorem_minimizer1}, we obtain that, for any $\alpha \in C^1(\mathbb{R}^n,\mathbb{R}^n)$ such that $J(\alpha,x_0)<+\infty$,
 \[\mathcal{E}(\rho_0) \leq \mathcal{E}(\rho_{\alpha}) \leq \tilde{J}(\alpha,\rho_{\alpha}) \leq \text{ess sup}_{x_0 \in \mathbb{R}^n} J(\alpha,x_0). \]
Taking the $\inf$ over $\alpha \in C^1(\mathbb{R}^n,\mathbb{R}^n)$ from the previous inequality we get $\mathcal{E}(\rho_0) \leq \mathfrak{J}$. 
\end{proof}

\begin{remark}\label{remark_mathfrakJ}
An important consequence of Theorem \eqref{theorem_optimalcontrol1} is that under Hypotheses $\mathcal{V}$ and C$\mathcal{V}$ we have that 
$$\mathfrak{J}=\mathcal{E}(\rho_0)=\inf_{\phi \in H^1(\mathbb{R}^n), \int \phi^2 dx=1}\mathcal{E}(\phi^2),$$
where $\mathfrak{J}$ is the value function associated with the problem \eqref{eq:SDEmain} and the cost functional \eqref{eq:J1}, defined by \eqref{eq:mathfrakJ}.
\end{remark}

\section{The $N$-particles approximation}\label{section_approximation}

In order to rigorously justify the limit McKean-Vlasov optimal control problem discussed in Section  \ref{section_existence} , in this section we  propose for it a natural many particles approximation. We consider the process $X_t=(X^1_t,...,X^N_t)\in \mathbb{R}^{nN}$ satisfying the SDE 
\begin{equation}\label{eq:SDEmain2}
dX^i_t=A^i_N(X_t)dt+\sqrt{2}dW^i_t,
\end{equation}
where $A_N=(A^1_N,...,A^N_N):\mathbb{R}^{nN} \rightarrow \mathbb{R}^{nN}$ is a $C^{1+\epsilon}$ function, for some $\epsilon>0$, and the $W^i_t, i=1,\dots, N$ are independent Brownian motions taking values in $\mathbb{R}^n$. 
\begin{remark}
It is important to note that, although the Brownian motions $W^i_t, i=1,\dots, N$ are independent, the processes $X^i_t,  i=1,\dots, N$ are in general not independent since we do not require that $A^i_N$ depends only on the variable $x^i$, but it can in general depend on all the variables $(x^1,...,x^n)$. 
\end{remark} 
If $\mathcal{V}$ is a functional satisfying Hypotheses $\mathcal{V}$, we introduce the functions sequence 
$$\mathcal{V}_N(x)=\sum_{i=1}^N \mathcal{V}\left(x_i,\frac{1}{N-1}\sum_{k=1,k\neq i}^N \delta_{x^i} \right),$$
where $x=(x^1,...,x^N) \in \mathbb{R}^{nN}$, $N \geq 2$.\\
We consider the (normalized with respect to the number of particles $N$) ergodic control problem
\begin{equation}\label{eq:J2}
J_N(A_N,x_0)=\limsup_{T \rightarrow +\infty}\frac{1}{N T} \int_0^T{\mathbb{E}_{x_0}\left[\frac{|A_N(X_t)|^2}{2}+\mathcal{V}_N(X_t)\right]dt}, 
\end{equation}  
and also the (normalized) energy functional 
\begin{equation}\label{eq:energy3}
\mathcal{E}_N(\rho_N)=\mathcal{E}_{K,N}(\rho_N)+\mathcal{E}_{P,N}(\rho_N)=\frac{1}{N}\left( 
\int_{\mathbb{R}^{nN}}{\frac{|\nabla \rho_N |^2}{2\rho_N}dx}+\int_{\mathbb{R}^{nN}}{\mathcal{V}_N(x)\rho_N(x)dx}\right),
\end{equation}
where $\rho_N$ is a positive Lebesgue integrable function such that $\int_{\mathbb{R}^{nN}}{\rho_N(x)dx}=1$. We also consider the value function
\begin{equation}\label{eq:mathfraJ2}
\mathfrak{J}_N=\text{ess sup}_{x_0 \in \mathbb{R}^n}\left(\inf_{A_N \in C^1(\mathbb{R}^{nN},\mathbb{R}^{nN})}J_N(A_N,x_0) \right).
\end{equation}
 Let us introduce the notation $\rho_N^{(1)}(x^1)=\int_{\mathbb{R}^{n(N-1)}}{\rho_N(x^1,x^2,...,x^N)dx^2...dx^N}$ for the one-particle probability density
and  let us finally put $\phi_N=\sqrt{\rho_N}$. \\
The next theorem, which is the analogue of Lemma \ref{theorem_minimizer1} for our $N$-particles control problem, gives important properties of the minimizer of the above energy functional. In particular, since the unique minimizer is symmetric, our N-particles control problem is intrinsically symmetric: for every fixed $N$ the diffusion components are not independent but they are identically distributed (see \cite{MU}).

\begin{lemma}\label{theorem_minimizer2}
Under Hypotheses $\mathcal{V}$, there exists a unique minimizer $\rho_{0,N}=\phi_{0,N}^2$ of the functional $\mathcal{E}_N$. This minimizer is symmetric in $x^1,...,x^N$, it is $C^{2+\epsilon}(\mathbb{R}^{nN})$, for some $\epsilon >0$, and it is strictly positive. Furthermore it is the only weak solution of the following linear PDE
\begin{equation}\label{eq:variational2}
-\Delta \phi_{0,N}(x)+2\mathcal{V}_N(x) \phi_{0,N}=\mu_N \phi_{0,N},
\end{equation}
where 
$$\mu_{N}=2\mathcal{E}_{N}(\phi_{0,N}^2).$$
\end{lemma}
\begin{proof}
Lemma \ref{theorem_minimizer2} can be seen as a special version of Lemma \ref{theorem_minimizer1} when $\mathcal{V}$ does not depend on $\rho$. The uniqueness of the minimizer is guaranteed by the fact that $\mathcal{E}_N(\phi^2)$ is quadratic with coefficients bounded from below (see, e.g., \cite{LiebLoss}, Chapter 11).
\end{proof}

\begin{remark}\label{remark_symmetry}
It is important to note that, by uniqueness of the minimizer $\rho_{0,N}$ of the functional $\mathcal{E}_N$, it follows that $\rho_{0,N}$ must be invariant with respect to coordinates permutations. Indeed it is simple to prove, using convexity of the Fisher information (see below), that if $\rho_{0,N}$ is a minimizer also its symmetrization is a minimizer (see, e.g., \cite{LiebLoss}, Chapter 7). An important consequence of the symmetry of $\rho_{0,N}$ with respect to coordinates permutations is that, when we consider $A_N=\frac{\nabla \rho_{0,N}}{\rho_{0,N}}$ and we start from a symmetric probability measure $p_N(x)$, as, for example, $\rho_{0,N}=p_N$ itself, then the process $(X^1_t,...,X^N_t)$ is \emph{symmetric with respect to permutations (or equivalently exchangeable)}. This observation plays a very important role in the rest of the paper.
\end{remark}
Finally the analogue of Theorem \ref{theorem_optimalcontrol1}  provides the optimal control.
\begin{theorem}\label{theorem_optimalcontrol2}
Under Hypotheses $\mathcal{V}$,  the logarithmic gradient  of the unique minimizer $\rho_{0,N}=\phi_{0,N}^2$ of $\mathcal{E}_N$, that is 
$$(A^1_N,...,A^N_N)=\left(\frac{\nabla_1 \rho_{0,N}}{\rho_{0,N}},...,\frac{\nabla_N \rho_{0,N}}{\rho_{0,N}}\right),$$ is the optimal control of the problem \eqref{eq:J2}.
\end{theorem}
\begin{proof}
Theorem \ref{theorem_optimalcontrol2} can be seen as a special version of Theorem \ref{theorem_optimalcontrol1} when $\mathcal{V}$ does not depend on $\rho$. 
\end{proof}

\begin{remark}
A very useful consequence of Theorem \ref{theorem_optimalcontrol2} is that 
$$\mathfrak{J}_N:=\mathcal{E}_N(\rho_{0,N})=\inf_{\phi_N \in H^1(\mathbb{R}^{nN}), \int \phi_N^2 dx=1}\mathcal{E}_N(\phi_N^2). $$
\end{remark}

\section{The convergence of value functions}\label{section_example}

In this section we prove the following convergence theorem.

\begin{theorem}\label{theorem_VE}
Suppose $\mathcal{V}$ satisfies Hypotheses $\mathcal{V}$ and C$\mathcal{V}$ then we have 
\begin{equation} \label{LimE1}
\lim_{N\rightarrow \infty} \mathfrak{J}_N=\mathfrak{J},
\end{equation}
where $\mathfrak{J}$ is as in \eqref{eq:mathfrakJ}. Furthermore we have
\begin{equation} \label{LimE2}
\lim_{N\rightarrow \infty} \mathcal{E}_{K,N}(\phi_{0,N}^2)=\mathcal{E}_K(\phi_{0}^2) 
\end{equation}
\begin{equation} \label{LimE3}
\lim_{N\rightarrow \infty} \rho_{0,N}^{(1)}(\cdot)= \lim_{N\rightarrow \infty} \int_{\mathbb{R}^{n(N-1)}}{\rho_{0,N}(\cdot,x^2,...,x^N)dx^2...dx^N}=\rho_0(\cdot)
\end{equation}
\noindent where the last limit is understood weakly in $L^1(\mathbb{R}^n,V(x)dx)$ and $\rho_{0}$ and $\rho_{0,N}$ are the unique minimizer of the functionals $\mathcal{E}$ and $\mathcal{E}_N$ respectively (with $\mathcal{E}$ given by \eqref{eq:energy21} and $\mathcal{E}_N$ given by \eqref{eq:energy3}).
\end{theorem}

\noindent Before proving this theorem we need to introduce some preliminary results.

\subsection{Some preliminary results}

In this section we recall de Finetti's theorem for exchangeable random variables in a setting that is useful for our aims and we discuss some related technical questions. Hereafter we use the following notation: If $\mu$ is a probability measure on $\mathbb{R}^n$ we denote by $\mu^{\otimes k}$ the probability measure on $\mathbb{R}^{nk}$ given by
$$\mu^{\otimes k}=\underbrace{\mu \otimes \cdots \otimes \mu}_{k \text{ times}}.$$
We adopt a similar notation for functions.

\begin{definition}
Let $\{\xi_i\}_{i \in \mathbb{N}}$ be a sequence of random variables such that each $ \xi_i$ lives in $\mathbb{R}^n$. We say that the sequence $\{\xi_i\}_{i \in \mathbb{N}}$ is exchangeable if for any finite permutation $\mathfrak{p}:\mathbb{N} \rightarrow \mathbb{N}$ we have that $\{ \xi_{\mathfrak{p}(i)}\}_{i \in \mathbb{N}}$ has the same joint probability law of $\{\xi_i\}_{i \in \mathbb{N}}$.
\end{definition}

\begin{proposition}[de Finetti theorem]\label{theorem_deFinetti}
Let $\{\xi_i\}_{i \in \mathbb{N}}$ be a sequence of  random variables on $\mathbb{R}^n$. They are exchangeable random variables if and only if there exists a random measure $\nu$ taking values on $\mathcal{P}(\mathbb{R}^n)$ such that 
$$\mathbb{P}[(\xi_{i_1},...,\xi_{i_k})|\nu]=\nu^{\otimes k},$$
for any $k \in \mathbb{N}$ and $i_1,...,i_k\in \mathbb{N}$ such that $i_j\not=i_{\ell}$.  $\mathbb{P}[(\xi_{i_1},...,\xi_{i_k})|\nu]$ is by definition the conditional probability law of $(\xi_{i_1},...,\xi_{i_k})$ given the random measure $\nu$.
\end{proposition}
\begin{proof}
The definitions and  the proof can be found in \cite[Theorem 1.1]{Kallenberg2005}.
\end{proof}

\begin{remark}\label{remark_deFinetti}
A consequence of the de Finetti theorem is the following. If $f:\mathbb{R}^{nk}\rightarrow \mathbb{R}$ is a bounded measurable function then 
$$\mathbb{E}[f(\xi_1,...,\xi_k)]=\int_{\mathcal{P}(\mathbb{R}^n)}{\int_{\mathbb{R}^{kn}}f(y_1,...,y_k)\mu(dy_1)\cdots \mu(dy_k)}\mathbb{P}_{\nu}(d\mu),$$
where $\mathbb{P}_{\nu}$ is the probability law of $\nu$ on $\mathcal{P}(\mathbb{R}^n)$. 
\end{remark}

De Finetti theorem is in general not true for \emph{finite sequences} $\{\xi_i^N \}_{i \leq N}$ of exchangeable random variables on $\mathbb{R}^n$. On the other hand we can take advantage of a limit result as follows. First we introduce the empirical measure associated with the finite sequence $\{\xi_i^N \}_{i \leq N}$ defined as:
$$\nu_N(dx)=\frac{1}{N}\sum_{i=1}^N\delta_{\xi_i^N}(dx).$$

\begin{proposition}\label{theorem_deFinetti2}
Let $\{\xi_i^N \}_{i \leq N}\in \mathbb{R}^{Nn}$ be a finite sequence of exchangeable random variables on $\mathbb{R}^n$.  The sequence $\{\xi_i^N \}_{i \leq N}$ converges in distribution to an infinite sequence of exchangeable random variables $\{\xi_i\}_{i\in\mathbb{N}}$ if and only if one of the following equivalent conditions hold  as $N\rightarrow \infty$:
\begin{enumerate}[i]
\item $(\xi^N_{1},...,\xi^N_k) \rightarrow (\xi_1,...,\xi_k)$ in distribution and  for any $k \in \mathbb{N}, k < N$,
\item $(\xi^N_{1},...,\xi^N_k,\nu_N) \rightarrow (\xi_1,...,\xi_k,\nu)$ in distribution and for any $k \in \mathbb{N}, k < N$.
\end{enumerate}
\end{proposition}
\begin{proof}
The proof can be found in \cite[Theorem 3.2]{Kallenberg2005}.
\end{proof}

Let us recall the definition of the Fisher information associated to a probability measure with density $\rho_N$ on $ \mathbb{R}^{nN}$ (see. e.g. \cite{HaMi}).

\begin{definition}\label{definition_Fisher}
For $\rho_N \in W^{1,1}(\mathbb{R}^{nN})$ we put
$$I_N(\rho_N):=\int_{\mathbb{R}^{nN}}\frac{|\nabla \rho_N|^2}{\rho_N}$$
otherwise we set $I_N(\rho_N) $ to be equal to $+\infty$.
We consider the normalized Fisher information $\mathcal{I}_N:=\frac{1}{N}I_N$
\end{definition}

Hereafter if $\rho_N$ is a probability density on $\mathbb{R}^{Nn}$ we denote by $\rho_N^{(k)}$ the projection of $\rho_N$ on the first $k$ coordinates namely 
$$\rho_N^{(k)}(x_1,...,x_k)=\int_{\mathbb{R}^{N-k}}{\rho_N(x_1,...,x_k,y_{k+1},...,y_N)dy_{k+1}\cdots dy_N}.$$

\begin{proposition}\label{theorem_fisher}
Let $\rho_N$ be a probability density on $\mathbb{R}^{Nn}$ which is invariant with respect to coordinates permutations, then we have:
\begin{enumerate}[i]
\item $I_N$ (and so $\mathcal{I}_N$) is a proper (in the sense of having compact sublevels), convex, lower semicontinuous ( l.s.c.)  functional (in the sense of the weak convergence of measures on $\mathcal{P}(\mathbb{R}^{nN}) $;
\item for $1\le \ell \le N, \quad \mathcal{I}_{\ell}(\rho^{(\ell)}_N) \le \mathcal{I}_N(\rho_N)$; 
\item the (non normalized) Fisher information is super-additive, i.e., for any $\ell=1,...,N$:
$$ I_N(\rho_N)\ge I_{\ell}(\rho^{(\ell)}_N)+ I_{N-l}(\rho^{(N-\ell)}_N)$$
with (in the case $I_{\ell}(\rho^{(\ell)}_N)+ I_{N-\ell}(\rho^{(N-\ell)}_N)<+\infty$) equality if and only if $\rho_N=\rho^{(\ell)}_N\rho^{(N-\ell)}_N$;
\item if $I(\rho^{(1)}_N) < +\infty$, the equality $\mathcal{I}_1(\rho^{(1)}_N)=\mathcal{I}_N(\rho_N)$ holds if and only if $\rho_N=(\rho^{(1)}_N)^{\otimes N}$.
\end{enumerate}
\end{proposition}
\begin{proof}
The proof can be found, e.g., in \cite{HaMi} Lemma 3.5, Lemma 3.6 and Lemma 3.7.
\end{proof}

We conclude this section by proving some useful results about the derivative of the infimum of a family of functions and about the derivatives of convex functions.

\begin{lemma}\label{theorem_inf}
Let $F:\mathcal{X} \times \mathbb{I} \rightarrow \mathbb{R}$ be a continuous function which is differentiable with respect to $t\in \mathbb{I}$, where $\mathbb{I} \subset \mathbb{R}$ is an open set and $\mathcal{X}$ is a metrizable compact space. Introducing $V(t)=\min_{x\in \mathcal{X}}F(x,t)$, $t\in \mathbb{I}$, let us suppose that $\sup_{(x,t)\in\mathcal{X}\times \mathbb{I}}|\partial_tF(x,t)|<+\infty$, that $\partial_tF$ is continuous, and that there exists a unique $x^*(t)$ such that $F(x^*(t),t)=V(t)$. Then the map $t \rightarrow x^*(t)$ is continuous, $V$ is $C^1(\mathbb{I})$ and
\begin{equation}\label{eq:derivative}
V'(t)=\partial_tF(x^*(t),t), \ t\in \mathbb{I}.
\end{equation}
\end{lemma}
\begin{proof}
 By Berge Maximum theorem (see, e.g., \cite[Theorem 17.31]{Aliprantis2006}) under the hypotheses of the theorem, $V$ is continuous and $x^*(t)$ is an upper semicontinuous correspondence. Since $x^*(t)$ is a single value correspondence (namely a function), this implies that $x^*(t)$ is continuous (see, e.g., \cite[Theorem 17.6]{Aliprantis2006}). On the other hand, by \cite[Theorem 3]{Milgrom2002}, we have that $F$ is right and left differentiable and 
$$V'_{\pm}(t_0)=\lim_{t\rightarrow t^{\pm}_0}\partial_tF(x^*(t),t_0).$$
Since both $\partial_tF$ (by hypothesis) and $x^*$ (as shown above) are continuous, we have that $V$ is differentiable and equation \eqref{eq:derivative} holds.
\end{proof}

\begin{lemma}\label{lemma_concave}
Let $F_n:\mathbb{I} \rightarrow \mathbb{R}$, where $\mathbb{I} \subset \mathbb{R}$ is an open set, be a sequence of $C^1(\mathbb{I})$ concave functions converging point-wise  as $n \rightarrow \infty$ to the $C^1(\mathbb{I})$ concave function $F:\mathbb{I} \rightarrow \mathbb{R}$. Then we have 
$$ \lim_{n \rightarrow +\infty} F'_{n}(t) = F'(t), \ t\in \mathbb{I}.  $$
\end{lemma}
\begin{proof}
For any $t_0 \in \mathbb{I}$ and  any $\epsilon >0$ there is $h_0 >0$ such that for any $0< h \leq h_0$ we have 
\begin{equation}\label{eq:concave1}
\frac{F(t_0)-F(t_0-h)}{h} \leq F'(t_0)+\epsilon. 
\end{equation}
On the other hand by the concavity of $F_n$ we have 
\begin{equation}\label{eq:concave2}
F'_n(t_0)\leq \frac{F_n(t_0)-F_n(t_0-h)}{h}.
\end{equation}
Taking the limit as $n \rightarrow +\infty$ in \eqref{eq:concave2} and introducing the result in \eqref{eq:concave1} we obtain $\limsup_{n \rightarrow +\infty} F_n'(t_0) \leq F'(t_0)+\epsilon$, that implies, by the arbitrary choice of $\epsilon$, \\$\limsup_{n \rightarrow +\infty}F_n'(t_0) \leq F'(t_0)$. Using a similar reasoning we are able to prove that $\liminf_{n \rightarrow +\infty}F_n'(t_0) \geq F'(t_0)$ from which we get the thesis.
\end{proof}

\subsection{Proof of Theorem \ref{theorem_VE}}

We start by proving three lemmas.
Let us denote by $\rho_{0,N}$ the probability density which is the minimizer of the function $\mathcal{E}_N(\rho)$ and let us consider a finite sequence of random variables $(\xi_1^N,...,\xi_N^N)\in \mathbb{R}^{Nn}$ having probability density $\rho_{0,N}$.

\begin{lemma}\label{lemma_fisher}
Under the hypotheses of Theorem \ref{theorem_VE} we have $\mathcal{E}_N(\rho_{0,N}) \leq \mathcal{E}(\rho_0)$ for any $N \in \mathbb{N}$. Furthermore the sequence $\{ \xi_i^N\}_{i \leq N}$ is a sequence of exchangeable random variables such that the corresponding sequence of probability distributions is tight and converges, as $N \rightarrow +\infty$, in distribution (up to passing to a subsequence) to some infinite sequence of exchangeable random variables $\{\xi_i\}_{i \in\mathbb{N}}$.
\end{lemma}
\begin{proof}
The first thesis of the lemma follows from the following inequalities
$$\mathcal{E}_N(\rho_{0,N})\leq \mathcal{E}_N(\rho^{\otimes N}_0) =\mathcal{E}(\rho_0),$$
where we used the fact that $\mathcal{E}_N(\rho^{\otimes N})=\mathcal{E}(\rho)$ for any probability density on $\mathbb{R}^n$. \\
First we note that by Remark \ref{remark_symmetry}, $\rho_{0,N}$ is unique and so it is symmetric with respect to permutations of coordinates. This means that $\{\xi_i^N\}_{i \leq N}$ are exchangeable random variables.  
We note that 
$$\mathcal{E}_N(\rho_{0,N})=\frac{1}{2}\mathcal{I}_N(\rho_{0,N})+\mathcal{E}_{P,N}(\rho_{0,N}) \geq \frac{1}{2}\mathcal{I}_N(\rho_{0,N}) - C, $$
for some constant $C\geq 0$, where we used that, by Hypothesis $\mathcal{V}$\emph{ii}, $\mathcal{V}(x,\mu)$ is uniformly bounded from below. Using Proposition \ref{theorem_fisher} \emph{ii} and the inequality $\mathcal{E}_N(\rho_{0,N})\leq \mathcal{E}(\rho_0)$ we have 
$$\mathcal{I}_k(\rho^{(k)}_{0,N})\leq\mathcal{I}_N(\rho_{0,N}) \leq 2 \mathcal{E}(\rho_0)+ 2 C .$$
By the fact that $\mathcal{I}_k$ is proper with respect to weak convergence of measures (see Proposition \ref{theorem_fisher} \emph{i}), we have that $\rho^{(k)}_N$ is a sequence of tight probability densities on $\mathbb{R}^{nk}$. Using a diagonalization argument there are a subsequence $N_j$ and a sequence of exchangeable and compatible probability measures $\mu_{\infty}^{(k)}$ on $\mathcal{P}(\mathbb{R}^{nk})$ (i.e. they are such that the restriction on the first $k$ coordinates of $\mu^{(k')}_{\infty}$ is exactly $\mu^{(k)}_{\infty}$ for any $k\leq k' \in \mathbb{N}$) such that as $ j \rightarrow +\infty$ 
$$\rho^{(k)}_{0,N_j}(y) dy \rightarrow \mu^{(k)}_{\infty}, \ y\in \mathbb{R}^{nk}$$
weakly. Since $\mu^{(k)}_{\infty}$ are compatible and invariant with respect to coordinates permutations, by Kolmogorov's extension theorem (see, e.g., \cite[Theorem 5.16]{Kallenberg2002}), there is a sequence of exchangeable random variables $\{ \xi_i \}_{i \in \mathbb{N}}$ such that $(\xi_1,...,\xi_k) $ has the law $\mu_{\infty}^{(k)}$. By Proposition \ref{theorem_deFinetti2} \emph{i}, $\{\xi^{N_j}_i\}_{i \leq N}$ is defined up to $N_j$ and converges in distribution to $\{\xi_i \}_{i \in \mathbb{N}}$, as $j \rightarrow +\infty$. 
\end{proof}

\begin{lemma}\label{lemma_convergencefunctional}
Under the hypotheses of Theorem \ref{theorem_VE}, we have $\mathfrak{J}_N \rightarrow \mathfrak{J}$, as $N \rightarrow +\infty$.
\end{lemma}
\begin{proof}
Since by Lemma \ref{lemma_fisher}, we have $\mathcal{E}_{N}(\rho_{0,N})\leq \mathcal{E}(\rho_0)$, in order to prove that $\mathcal{E}_N(\rho_{0,N}) \rightarrow \mathcal{E}(\rho_0)$  as $N \rightarrow \infty$  it is sufficient to establish a lower bound for $\liminf_{N \rightarrow+ \infty}\mathcal{E}_N(\rho_{0,N})$. Passing to a suitable subsequence we can suppose that 
$$ \liminf_{N \rightarrow+ \infty}\mathcal{E}_N(\rho_{0,N})=\lim_{N \rightarrow +\infty}\mathcal{E}_N(\rho_{0,N}).$$
Let $\xi_i^N$ and $\xi_i$ be as in Lemma \ref{lemma_fisher}. Then by Lemma \ref{lemma_fisher}, by Proposition \ref{theorem_deFinetti2} \emph{ii}, by Skorohod representation theorem (see, e.g. \cite[Theorem 3.2]{Kallenberg2002}) and using an abuse of notation identifying the subsequence with the whole sequence, we can suppose that $(\{ \xi_i^N \}_{i \leq N},\nu_N)$ converges to $(\{\xi_i\},\nu)$ almost surely, as $N \rightarrow +\infty$.
We have that 
$$\mathcal{E}_N(\rho_{0,N})=\frac{1}{2}\mathcal{I}_N(\rho_{0,N})+\mathbb{E}[\mathcal{V}(\xi_1^N,\tilde{\nu}_N)]$$
where $\tilde{\nu}_N=\frac{1}{N-1}\sum_{2\leq i \leq N}\delta_{\xi_i^N}$. By Proposition \ref{theorem_fisher} \emph{ii} and lower semicontinuity of Fisher information (see Proposition \ref{theorem_fisher} \emph{i}) we have that  
\begin{equation}\label{eq:convergencefunctional1}
\liminf_{N \rightarrow +\infty}\mathcal{I}_N(\rho_{0,N})\geq \liminf_{N \rightarrow +\infty}\mathcal{I}_1(\rho^{(1)}_{0,N})\geq \mathcal{I}_1(\law(\xi_1))=\mathcal{E}_K(\mathbb{E}[\nu]).
\end{equation}
Since $\tilde{\nu}_N-\nu_N$ converges to $0$ in total variation, by Fatou lemma, Hypothesis ${\mathcal{V}}$\emph{i} and Jensen inequality, we have that 
\begin{align}
\liminf_{N \rightarrow +\infty}\mathbb{E}[\mathcal{V}(\xi_1^N,\tilde{\nu}_N)]\geq &\mathbb{E}[\liminf_{N \rightarrow +\infty}\mathcal{V}(\xi_1^N,\tilde{\nu}_N)]\nonumber\\
\geq&\mathbb{E}[\mathcal{V}(\xi_1,\nu)]=\mathbb{E}[\mathbb{E}[\mathcal{V}(\xi_1,\nu)|\nu]]\nonumber\\
=&\mathbb{E}[\tilde{\mathcal{V}}(\nu)]\geq\tilde{\mathcal{V}}(\mathbb{E}[\nu])=\mathcal{E}_P(\mathbb{E}[\nu]) \label{eq:convergencefunctional2}
\end{align}
From the previous inequalities and the fact that  $\rho_0$ is the minimizer, we obtain that 
$$\lim_{N \rightarrow + \infty}\mathcal{E}_{N}(\rho_{0,N}) \geq \mathcal{E}(\mathbb{E}[\nu]) \geq \mathcal{E}(\rho_0),$$
and this concludes the proof of Lemma \ref{lemma_convergencefunctional}.
\end{proof}

Let us return to the proof of Theorem \ref{theorem_VE}. The first statement \eqref{LimE1} is an immediate consequence of Lemma \ref{lemma_convergencefunctional}. In order to prove the remaining relations \eqref{LimE2} and \eqref{LimE3} we want to use a variational argument proposed, for example, in \cite{CS} (see also \cite{Strong}). The main idea is to introduce some modified functionals $\mathfrak{E}(\lambda,\phi^2)$, $\mathfrak{E}'(\lambda',\phi^2,f)$, $\mathfrak{E}_N(\lambda,\phi^2)$ and $\mathfrak{E}'_N(\lambda',\phi^2,f)$ (where $\lambda$, $\lambda'$ are real parameters and $f$ is a suitable smooth function) such that $\mathfrak{E}(1,\phi^2)=\mathcal{E}(\phi^2)$, $\mathfrak{E}'(0,\phi^2,f)=\mathcal{E}(\phi^2)$, $\mathfrak{E}_N(1,\phi^2)=\mathcal{E}_N(\phi^2)$ and  $\mathfrak{E}'_N(0,\phi^2,f)=\mathcal{E}_N(\phi^2)$, and also the derivatives with respect to $\lambda$ and $\lambda'$ are equal to the expressions involved in equations \eqref{LimE2} and \eqref{LimE3}. We will prove that $\mathfrak{E}_N\rightarrow \mathfrak{E}$ and $\mathfrak{E}'_N \rightarrow \mathfrak{E}'$ for any $\lambda$, $\lambda'$ and $f$ in suitable sets, as $N\rightarrow+\infty$. This implies the convergence of the derivatives of $\mathfrak{E}_N$ and $\mathfrak{E}'_N$ with respect to $\lambda$ and $\lambda'$ giving us the limits \eqref{LimE2} and \eqref{LimE3}.\\
More precisely, we introduce a little modification of the functionals $\mathcal{E}_{N}$ and $\mathcal{E}$ by writing 
\begin{align*}
\mathfrak{E}(\lambda,\phi^2)=&\int_{\mathbb{R}^{n}}{\left(\frac{|\nabla \phi(x)|^2}{2}+\lambda\mathcal{V}(x,\phi^2)\phi^2(x)\right)dx} \\
\mathfrak{E}'(\lambda',\phi^2,f)=&\int_{\mathbb{R}^{n}}{\left(\frac{|\nabla \phi(x)|^2}{2}+\mathcal{V}(x,\phi^2)\phi^2(x)+\lambda' f(x)V(x)\phi^2(x)\right)dx}\\
\mathfrak{E}_{N}(\lambda,\phi^2_N)=&\frac{1}{N}\int_{\mathbb{R}^{Nn}}{\left(\frac{|\nabla \phi(x)|^2}{2}+\lambda\mathcal{V}_N(x)\phi^2_N(x)\right)dx} \\
\mathfrak{E}'_{N}(\lambda',\phi^2_N,f)=&\frac{1}{N}\int_{\mathbb{R}^{Nn}}{\left(\frac{|\nabla \phi(x)|^2}{2}+\mathcal{V}_N(x)\phi^2_N(x)+\lambda' \sum_{i=1}^Nf(x_i)V(x_i)\phi^2(x)\right)dx},
\end{align*}
where $f \in C^{\infty}_b(\mathbb{R}^n)$ and such that $\| f\|_{\infty}\leq 1$. If $\lambda \in \mathbb{I} \subset \mathbb{R}$ and $\lambda' \in \mathbb{I}'$, where $\mathbb{I}$ and $\mathbb{I}'$ are small enough neighborhoods of $1$ and $0$ respectively, $\lambda \mathcal{V}$, $\lambda \mathcal{V}_N$, $\mathcal{V}+\lambda' fV$ and $\mathcal{V}+\lambda'\sum_i fV$ satisfy hypotheses $\mathcal{V}$ and C$\mathcal{V}$ whenever $ \mathcal{V}$ and $ \mathcal{V}_N$ satisfy hypotheses $\mathcal{V}$ and C$\mathcal{V}$. By Lemma \ref{theorem_minimizer1} and Lemma \ref{theorem_minimizer2}, this means that there exist some uniquely determined positive functions $\phi^{\lambda}_0,\phi'^{\lambda'}_0\in H^1(\mathbb{R}^n) \cap L^2(\mathbb{R}^n,V(x)dx)$ and $\phi^{\lambda}_{0,N},\phi'^{\lambda}_{0,N}\in H^1(\mathbb{R}^n) \cap L^2(\mathbb{R}^{nN},\sum_{i=1}^NV(x_i)dx)$ which are the minimizers of $\mathfrak{E}(\lambda,\cdot^2)$, $\mathfrak{E}'(\lambda',\cdot^2)$, $\mathfrak{E}_N(\lambda,\cdot^2)$ and $\mathfrak{E}'_N(\lambda',\cdot^2)$ under the conditions, respectively, $\int_{\mathbb{R}^n}{\phi^{\lambda}_0(x)^2dx}=1$, $\int_{\mathbb{R}^n}{\phi'^{\lambda'}_0(x)^2dx}=1$ , $\int_{\mathbb{R}^{nN}}{\phi^{\lambda}_{0,N}(x)^2dx}=1$ and $\int_{\mathbb{R}^{nN}}{\phi'^{\lambda'}_{0,N}(x)^2dx}=1$.

\begin{lemma}\label{lemma_compact}
Under the hypotheses of Theorem \ref{theorem_VE}, there are $\mathcal{X}$, $\mathcal{X}'$, $\mathcal{X}_N$ and $\mathcal{X}'_N$, that are compact subsets of $H^1(\mathbb{R}^n) \cap L^2(\mathbb{R}^n,V(x)dx)$ and $H^1(\mathbb{R}^{nN}) \cap L^2(\mathbb{R}^{nN},\sum_{i=1}^NV(x_i)dx)$ respectively, such that $\phi^{\lambda}_0 \in \mathcal{X}$,  $\phi'^{\lambda'}_0 \in \mathcal{X}'$, $\phi^{\lambda}_{0,N} \in \mathcal{X}_N$ and $\phi'^{\lambda'}_{0,N} \in \mathcal{X}'_N$ for any $\lambda \in \mathbb{I}$ and $\lambda' \in \mathbb{I}'$ (with $\phi^{\lambda}_0$ and $\phi'^{\lambda'}_0$ as defined just before the statement of this lemma).
\end{lemma}
\begin{proof}
We give the proof only for $\phi^{\lambda}_{0}$, the proof for $\phi'^{\lambda'}_0$, $\phi^{\lambda}_{0,N}$ and $\phi'^{\lambda'}_{0,N}$ being completely analogous.\\
By Lemma \ref{theorem_minimizer1} we have that $\phi^{\lambda}_0$ satisfies the equation 
\begin{equation}\label{eq:compact1}
-\Delta \phi_0^{\lambda}(x)+\lambda V(x)\phi_0^{\lambda}(x) = \mu_{0,\lambda} \phi_0^{\lambda}(x)- \lambda \left(2(\partial_{\mu}\tilde{\mathcal{V}})(x,(\phi^{\lambda}_0)^2)-V(x)\right)\phi_0^{\lambda}(x),
\end{equation}
where $\mu_{0,\lambda}$ is given by expression \eqref{eq:mu0}. By Hypotheses $\mathcal{V}$\emph{ii} and $\mathcal{V}$\emph{iii} we have that  $\left(2(\partial_{\mu}\tilde{\mathcal{V}})(x,(\phi^{\lambda}_0)^2)-V(x)\right)$ is bounded from below. Writing 
$$E=\sup_{\lambda \in \mathbb{I}}\mathfrak{E}(\lambda,(\phi^{\lambda}_0)^2), $$
which is finite for $\mathbb{I}$ small enough, by multiplying equation \eqref{eq:compact1} by $V(x)\phi^{\lambda}_0$ and integrating, using integration by parts and formula \eqref{eq:mu0}, we obtain
\begin{equation}\label{eq:spiegazione}\int_{\mathbb{R}^n}V(x)|\nabla \phi^{\lambda}_0(x) |^2dx+\int_{\mathbb{R}^n}\phi^{\lambda}_0(x)(\nabla V(x) \cdot\nabla \phi^{\lambda}_0(x) )dx\\+\lambda \int_{\mathbb{R}^n}{(V(x)\phi_0^{\lambda}(x))^2dx}
\leq E+C_{V,\mathcal{V}}
\end{equation}
for some constant $C_{V,\mathcal{V}}$ depending on $V$. Exploiting the properties \eqref{eq:conditionV} for $V$, a weighted Young inequality on $\phi^{\lambda}_0| \nabla \phi^{\lambda}_0|$, the fact that $-k V(x)+V(x)^2\geq k' V(x)^2-k''$ for any $k \in \mathbb{R}_+$ and some $k',k''$ depending on $k$ and $V$, and multiplying both sides of \eqref{eq:spiegazione} by a suitable constant we obtain that 
\begin{equation}\label{eq:compact2}
\int_{\mathbb{R}^n}{V(x)\left(|\nabla\phi^{\lambda}_0(x)|^2+V(x)(\phi^{\lambda}_0(x))^2\right)dx} \leq C_{\mathbb{I},V,\mathcal{V}} (E+1),
\end{equation} 
where $C_{\mathbb{I},V,\mathcal{V}}$ is a positive constant depending only on $\mathbb{I}$, $V$ and $\mathcal{V}$. By multiplying equation \eqref{eq:compact1} by $V(x)^2 \phi^{\lambda}_0$ and $V(x)\Delta \phi^{\lambda}_0$, using a similar reasoning and inequality \eqref{eq:compact2} we obtain 
\begin{align*}
&\int_{\mathbb{R}^n}{V(x)^2\left(|\nabla\phi^{\lambda}_0(x)|^2+V(x)(\phi^{\lambda}_0(x))^2\right)dx} \leq  C'_{\mathbb{I},V,\mathcal{V}} (E+1)\\
&\int_{\mathbb{R}^n}{V(x)(\Delta \phi^{\lambda}_0(x))^2dx} \leq  C''_{\mathbb{I},V,\mathcal{V}} (E+1)^2,
\end{align*}
for some positive constants $C'_{\mathbb{I},V,\mathcal{V}}$, $C''_{\mathbb{I},V,\mathcal{V}}$ depending only on $\mathbb{I}$, $V$ and $\mathcal{V}$.\\
Using the fact that, by the properties \eqref{eq:conditionV} of $V$, $\int_{\mathbb{R}^n}{V(x)((\Delta\phi(x))^2+V(x) \phi(x)^2)dx}$ is an equivalent norm of $H^2(\mathbb{R}^n,V(x)dx) \cap L^2(\mathbb{R}^n,V(x)^2dx)$ (see, e.g. \cite[Section 5.1.5]{Triebel1987} where this assertion is proven for more general Besov spaces, see also \cite{Schott1,Schott2}) we get that $\phi^{\lambda}_0$ is contained in some bounded subset $\mathcal{X}$ of $H^2(\mathbb{R}^n,V(x)dx) \cap L^2(\mathbb{R}^n,V(x)^2dx)$ . Since $V$ grows to $+\infty$ when $|x| \rightarrow +\infty$, the embedding of $H^2(\mathbb{R}^n,V(x)dx) \cap L^2(\mathbb{R}^n,V(x)^2dx)$ in  $H^1(\mathbb{R}^n) \cap L^2(\mathbb{R}^n,V(x)dx)$ is compact which implies that $\mathcal{X}$ is compact in $H^1(\mathbb{R}^n) \cap L^2(\mathbb{R}^n,V(x)dx)$.
\end{proof}

\begin{proof}[Proof of the equalities \eqref{LimE2} and \eqref{LimE3} in  Theorem  \ref{theorem_VE}]
 We are now able to prove equation  \eqref{LimE2} by establishing that $\mathcal{E}_{P,N}(\rho_{0,N}) \rightarrow \mathcal{E}_P(\rho_0)$, as $N \rightarrow +\infty$. We introduce the functions
\begin{eqnarray*}
&E_N(\lambda)=\mathfrak{E}_N(\lambda,(\phi^{\lambda}_{0,N})^2)=\min_{\phi_N \in \mathcal{X}_N}\mathfrak{E}_N(\lambda,(\phi_{N})^2) &\\
&E(\lambda)=\mathfrak{E}(\lambda,(\phi^{\lambda}_{0})^2)=\min_{\phi \in \mathcal{X}}\mathfrak{E}(\lambda,\phi^2), &
\end{eqnarray*}
where $\mathcal{X}_N$ and $\mathcal{X}$ are the compact sets built in Lemma \ref{lemma_compact}. By Lemma \ref{lemma_convergencefunctional} we have that $E_N(\lambda) \rightarrow E(\lambda)$ for $\lambda$ in a neighborhood $\mathbb{I}$ of $1$ small enough, as $N\rightarrow +\infty$. Furthermore, since, by Lemma \ref{lemma_compact}, we have that $\mathcal{X}_N$ and $\mathcal{X}$ are compact metrizable sets,  we can apply Lemma \ref{theorem_inf} to $E_N$ and $E$ getting respectively
$$\partial_{\lambda}E_N(1)=\mathcal{E}_{P,N}(\rho_{0,N}) \quad \quad  \partial_{\lambda}E(1)=\mathcal{E}_{P}(\rho_{0}).$$
On the other hand, since $\mathfrak{E}(\lambda,\phi)$ and $\mathfrak{E}_N(\lambda,\phi_N)$ are affine functions in $\lambda$,  we have that $E$ and $E_N$ are concave functions, being the minimum of concave functions. This means that, by Lemma \ref{lemma_concave}, $\partial_{\lambda}E_N(1) \rightarrow \partial_{\lambda}E(1)$, thus proving  the equalities \eqref{LimE2}.

Applying a similar reasoning to 
\begin{eqnarray*}
&E_N'(\lambda,f)=\mathfrak{E}_N(\lambda,(\phi^{\lambda}_{0,N})^2,f)=\min_{\phi_N \in \mathcal{X}'_N}\mathfrak{E}_N(\lambda,(\phi_{N})^2,f) &\\
&E'(\lambda,f)=\mathfrak{E}(\lambda,(\phi^{\lambda}_{0})^2)=\min_{\phi \in \mathcal{X}'}\mathfrak{E}'(\lambda,\phi^2,f), &
\end{eqnarray*}
we prove that
$$\int_{\mathbb{R}^n}{V(x)f(x)\rho^{(1)}_{0,N}(x)dx} \rightarrow \int_{\mathbb{R}^n}{V(x)f(x)\rho_{0}(x)dx}, $$
as $N \rightarrow +\infty$. Since $f$ is any $C^{\infty}(\mathbb{R}^n)$ bounded function we have that $\rho^{(1)}_{0,N}$ converges to $\rho_0$ weakly in $L^1(\mathbb{R}^n,V(x)dx)$, as $N \rightarrow +\infty$, and this implies \eqref{LimE3}. This complete the proof of Theorem \ref{theorem_VE}.
\end{proof}

\section{Convergence of the invariant measures}\label{section_fixedtime}

For positive $L^1(\mathbb{R}^{kn},dx)$ functions $\rho_1,\rho_2$ define 
$$H_k(\rho_1|\rho_2):=\left\{ \begin{array}{ll}
\int_{\mathbb{R}^{kn}}\log\left(\frac{\rho_1(x)}{\rho_2(x)}\right)\rho_1(x)dx & \text{if supp}(\rho_1) \subset \text{supp}(\rho_2)\\
+\infty & \text{ elsewhere}
\end{array}\right. , $$
and also, we define, with an abuse of notation, $H_k(\rho_1):=H_k(\rho_1|1)$.\\

The main aim of the present section is to prove the following theorem.

\begin{theorem}\label{theorem_HH1}
Under Hypotheses $\mathcal{V}$, C$\mathcal{V}$ and Q$V$ and using the notations of Theorem \ref{theorem_VE}, we have
\begin{equation}\label{eq:HH1}
\frac{1}{N}H_N(\rho_{0,N}|\rho_0^{\otimes N})\rightarrow 0, 
\end{equation}
as $N\rightarrow +\infty$. 
\end{theorem}

\begin{remark}\label{remark_HH1}
Thanks to Lemma \ref{lemma3} below, Theorem \ref{theorem_HH1} implies that 
$$H_k(\rho_{0,N}^{(k)}|\rho_0^{\otimes k})\rightarrow 0$$
as $N\rightarrow +\infty$. 
\end{remark}

The proof of Theorem \ref{theorem_HH1} itself will be given in Section \ref{subsection_proof} below. First we prove the Kac's chaoticity of the sequence $\rho_{0,N}$.

\subsection{Kac's chaoticity of the sequence $\rho_{0,N}$}

In this subsection we prove that the sequence of symmetric measures $\mu_{0,N}(dx_1\cdots dx_N)=\rho_{0,N}(x)dx_1\cdots dx_N$ is Kac's chaotic with limit $\mu_0(dx)=\rho_0(x)dx$ in the sense of the following definition.

\begin{definition}
Let $\mu_1,...,\mu_N,...$ be a sequence of symmetric probability measures on $\mathbb{R}^n,...,\mathbb{R}^{nN},...$ respectively, denote by $\mu^{(k)}_K$, for $k\leq K\in \mathbb{N}$, the projection of $\mu_K$ on $\mathbb{R}^{nk}$, and let $\mu$ be a probability measure on $\mathbb{R}^n$. We say that $\mu_1,...,\mu_N,...$ is Kac's chaotic with limit $\mu$ if $\mu_{N}^{(k)}$ weakly converges to $\mu^{\otimes k}$ on $\mathbb{R}^{nk}$ as $N\rightarrow+\infty$. 
\end{definition}

\begin{theorem}\label{theorem_kac}
Under Hypotheses $\mathcal{V}$ and C$\mathcal{V}$ we have that the sequence $\{\mu_{0,N}\}_{N\in \mathbb{N}}$ is Kac's chaotic with limit $\mu_0$. 
\end{theorem}
\begin{proof}
Here we use the concepts and notations introduced in Section \ref{section_example}. \\
Since the measures $\mu_{0,N}^{(k)}$ are tight by Lemma \ref{lemma_fisher}, the statement of the theorem is equivalent to proving that any convergent subsequence of $\xi^N=(\xi_1^N,...,\xi^N_N)$ converges to a sequence of independent identically distributed random variables $(\xi_1,...,\xi_k,...)$ having probability law $\mu_0$. \\
By Proposition \ref{theorem_deFinetti} and Remark \ref{remark_deFinetti}, this is equivalent to proving that $\nu=\mu_0$ almost surely or equivalently $\mathbb{P}_{\nu}=\delta_{\mu_0}$, where $\nu$  is the measure related to the limit exchangeable sequence  $(\xi_1,...,\xi_k,...)$ by Proposition \ref{theorem_deFinetti}.\\
By Lemma \ref{lemma_fisher}, the statement and the proof (in particular inequalities \eqref{eq:convergencefunctional1} and \eqref{eq:convergencefunctional2}) of Lemma \ref{lemma_convergencefunctional} and Proposition \ref{theorem_fisher} \emph{iii}  we have
\begin{eqnarray}
 \mathcal{E}(\rho_0)&\geq& \mathcal{I}_{k}\left(\mathbb{E}[\nu^{\otimes k}] \right)+\mathcal{E}_K(\mathbb{E}[\nu])\nonumber\\
 &\geq & \mathcal{I}_{1}\left(\mathbb{E}[\nu] \right)+\mathcal{E}_K(\mathbb{E}[\nu])=\mathcal{E}(\rho_0),\label{eq:kac1}
\end{eqnarray}
Since, by Theorem \ref{theorem_VE}, $\mathbb{E}[\nu]=\mu_0$, inequality \eqref{eq:kac1} implies that 
\begin{equation}\label{eq:kac2}
\mathcal{I}_{k}(\mathcal{E}[\nu^{\otimes k}])=\mathcal{I}_{1}(\mathcal{E}[\nu])=\mathcal{I}_{1}(\mu_0).
\end{equation}
On the other hand, by Proposition \ref{theorem_fisher} \emph{iii}, from equation \eqref{eq:kac2} we get 
\begin{equation}\label{eq:kac3}
\mathbb{E}[\nu^{\otimes k}]=\mu_0^{\otimes k}.
\end{equation}
Since relation \eqref{eq:kac3} is true for any $k\in \mathbb{N}$ we must have $\nu=\mu_0$ almost surely, which implies the thesis of the theorem.
\end{proof}

\subsection{Proof of Theorem \ref{theorem_HH1}}\label{subsection_proof}

In order to prove Theorem \ref{theorem_HH1} we introduce some concepts and some preliminary lemmas. 

\begin{lemma}\label{lemma_rho}
Under hypotheses $\mathcal{V}$, C$\mathcal{V}$ and Q$V$ on $\mathcal{V}$, $\tilde{\mathcal{V}}$ and $V$, we have that there is $R>0$ and some constants $a_1,a_2\in \mathbb{R}_+$ and $a_3,a_4\in\mathbb{R}$ such that for any $x\in \mathbb{R}$ for which $|x|\geq R$ we have 
\begin{equation}\label{eq:rho}
\exp\left(-a_1\int_0^{|x|}\sqrt{\bar{V}(r)}dr+a_3 \right) \leq \rho_0(x) \leq \exp\left(-a_2 |x|^2 +a_4 \right).
\end{equation}
\end{lemma}

In order to prove the previous Lemma \ref{lemma_rho}, we use the following Proposition \ref{theorem_rho}.

\begin{proposition}\label{theorem_rho}
Let $B_R\subset \mathbb{R}^n$ be the closed ball of radius $R>0$ and center $0\in \mathbb{R}^n$. Let $f,g$ be functions smooth in $\overline{\mathbb{R}^n \backslash B_R}$ and let $K,H$ be measurable functions defined on $\overline{\mathbb{R}^n \backslash B_R}$  such that:
\begin{enumerate}[i]
\item $\Delta |f| \leq K |f|$,
\item $\Delta |g| \geq H |g|$,
\item $f,g \rightarrow 0$ as $|x| \rightarrow +\infty$,
\item $H(x) \geq K(x) \geq 0$ for any $x\not \in B_R$,
\item $|f(x)|\geq |g(x)|$ for any $x \in\partial B_R$,
\end{enumerate}
then we have $|f(x)| \geq |g(x)|$ for all $x\in \overline{\mathbb{R}^n \backslash B_R}$. 
\end{proposition}
\begin{proof}
The proof can be found in \cite{Simonbound} Theorem 8.
\end{proof}

\begin{proof}[Proof of Lemma \ref{lemma_rho}]
We prove only the left part of the relation \eqref{eq:rho}, since the proof of the right part can be done in a similar way using Hypothesis Q$V$\emph{i} (furthermore the complete proof of  the right part can also be found in \cite{Simonbound}).\\
By Hypotheses $\mathcal{V}$ and Lemma \ref{theorem_minimizer1} we have 
$$\Delta \phi_0(x) \leq C_1 V(x) \phi_0(x) $$
for a suitable constant $C_1 >0$ and for $|x|\geq R$ for some $R\geq 0$. Writing $F_{a'_1,a'_3}(x)=\exp\left(-a'_1\int_0^{|x|}\sqrt{\bar{V}(r)}dr+a'_3 \right)$ and choosing $a'_1>2(e_3+2(n-1))$ (where $e_3$ is the same constant in Hypothesis Q$V$) by Hypothesis Q$V$\emph{ii} we have 
\begin{eqnarray}
\Delta F_{a'_1,a'_3}(x) &= & a'_1 F_{a'_1,a'_3}(x)\frac{\left(-|x|\bar{V}'(|x|)+a'_1|x|(\bar{V}(|x|))^{\frac{3}{2}}-2(n-1)\bar{V}(|x|) \right)}{2|x|\sqrt{\bar{V}(|x|)}}\\
&\geq & \frac{a'_1}{4} F_{a'_1,a'_3}(x) V(x).
\end{eqnarray}
Choosing $\frac{a'_1}{4}>C_3$, since, by Lemma \ref{theorem_minimizer1}, $\phi_0>0$, there exists an $\bar{a}_3>0$ (depending on $R$ and on $a'_1$) such that 
$$\sup_{x\in \partial B_R}  F_{a'_1,\bar{a}_3}(x) \leq \sup_{x\in \partial B_R}\phi_0(x).$$
By Proposition \ref{theorem_rho} we have $\phi_0(x)\geq F_{a'_1,\bar{a}_3}(x)$ for $x\in \overline{\mathbb{R}^n \backslash B_R}$, which implies that $F_{2a'_1,2\bar{a}_3}(x)=F_{a_1,a_3}(x) \leq \rho_0(x)$ for $x\in \overline{\mathbb{R}^n \backslash B_R}$.
\end{proof}

\begin{lemma}\label{lemma_Hentropy1}
Under Hypotheses $\mathcal{V}$, C$\mathcal{V}$ and Q$V$ we have that $|H_{nN}(\rho_{0,N})|,|H_n(\rho_0)|<+\infty$ and we have the convergence
\begin{equation}\label{eq:entropy1}
\frac{1}{N}H_{N}(\rho_{0,N}) \rightarrow H_1(\rho_0) 
\end{equation}
as $N \rightarrow +\infty$.
\end{lemma}
\begin{proof}
The fact that  $|H_{nN}(\rho_{0,N})|$ is finite follows from the logarithmic Sobolev inequality for the Lebesgue measure (see, i.e., Section 4.6.1 of \cite{Sobolevinequalities}) in fact we have
$$0\leq \frac{2}{nN} H_{N}(\rho_{0,N})\leq \log\left( \frac{2}{nN e \pi} \int_{\mathbb{R}^{nN}}\frac{|\nabla\rho_{0,N}(x)|^2}{\rho_{0,N}(x)}dx\right)= \log\left( \frac{2 }{n e \pi} \mathcal{I}_N(\rho_{0,N})\right),$$
where $\mathcal{I}_N$ is the normalized Fisher information introduced in Section \ref{section_example}. A similar inequality holds also for $|H_1(\rho_0)|$.\\
The convergence \eqref{eq:entropy1} follows from the following facts
\begin{enumerate}[i] 
\item by Theorem \ref{theorem_kac}, the sequence $\mu_{0,N}(dx)=\rho_{0,N}(x)dx$ is Kac's chaotic and converging to $\mu_0(dx)=\rho_0(x)dx$, as $N \rightarrow +\infty$;
\item by Theorem \ref{theorem_VE}, in particular equation \eqref{LimE2}, we get $\sup_{N\in \mathbb{N}} \mathcal{I}_{N}(\rho_{0,N})<+\infty$;
\item by Hypothesis Q$V$, we have $\sup_{N\in \mathbb{N}}\int_{\mathbb{R}^n}{\frac{\|x\|^{2+\epsilon}}{N}\rho_{0,N}(x)dx}<+\infty$.
\end{enumerate}
Indeed the previous properties i, ii and iii of $\mu_{N,0}$ are the hypotheses of Theorem 1.4 of \cite{HaMi}, which implies the statement \eqref{eq:entropy1}.
\end{proof}

\begin{proof}[Proof of Theorem \ref{theorem_HH1}]
By Hypothesis Q$V$\emph{i}, we have that, for every $x\in \mathbb{R}^n$,
\begin{equation}\label{eq:barV}
\int_0^{|x|}\sqrt{\bar{V}(r)}dr\leq |x| \sqrt{\bar{V}(|x|)} \leq \sqrt{\frac{V(x)+e_2}{e_1}}\sqrt{V(x)}\leq  P_1 V(x)+P_2,
\end{equation} 
for suitable constants $P_1,P_2>0$, for all $x\in \mathbb{R}^n$. Furthermore by Lemma \ref{lemma_rho}, inequality \eqref{eq:barV} and Q$V$\emph{i} we get 
\begin{equation}\label{eq:barV2}
|\log(\rho_0(x))|\leq a_1 V(x)+ a_3 +a_2 |x|^2 +a_4 \leq P_3 V(x) +P_4
\end{equation}
for suitable constants $P_3,P_4>0$.  
By Lemma \ref{lemma_Hentropy1} $|H_{nN}(\rho_{0,N})|$ is bounded and using inequality \eqref{eq:barV2}, we obtain
$$\int_{\mathbb{R}^n}{|\log(\rho_0(x))|\rho_{0,N}^{(1)}(x)dx}\leq P_3\int_{\mathbb{R}^n}{V(x)\rho_{0,N}^{(1)}(x)dx}+P_4<+\infty.$$ 
Thus we get that 
\begin{equation}\label{eq:Hentropy1}
\frac{1}{N}H_{N}(\rho_{0,N}|\rho_0^{\otimes N}) =\frac{1}{N} H_{N}(\rho_{0,N})-\int_{\mathbb{R}^n}{\log(\rho_0(x))\rho_{0,N}^{(1)}(x)dx}.
\end{equation}
Using Lemma \ref{lemma_Hentropy1}, inequality \eqref{eq:barV2} and Theorem \ref{theorem_VE}, the thesis follows.
\end{proof}

\begin{remark}\label{remark_QV}
Hypothesis Q$V$ is used crucially in two points of the proof of Theorem \ref{theorem_HH1}. The first time it is used in Lemma \ref{lemma_rho} (and for a similar reason in inequality \eqref{eq:barV}) in order to be able to control the growth of $|\log(\rho_0)|$ at infinity by the function $V(x)$. The second time, Hypothesis Q$V$ is exploited in the proof of Lemma \ref{lemma_Hentropy1}, in particular at the place where Theorem 1.4 of \cite{HaMi} is cited. Indeed, the proof of Theorem 1.4 of \cite{HaMi} uses in an essential way the HWI inequalities (see \cite{Villani} Chapter 20) controlling the relative entropy of two measures by their Wasserstein $W_{2,\mathbb{R}^n}$  distance (see Definition \ref{definition_wesserstein}), which is finite when the second moment exists, and their Fisher information.
\end{remark}

\section{Convergence of the probability law on the path space}\label{section_convergence}

In this section we prove the convergence on the path space of the $N$-particles system control problem \eqref{eq:J2}, when the initial condition is the invariant measure $\rho_{0,N}$, as $N \rightarrow + \infty$, to the McKean-Vlasov optimal control problem given by \eqref{eq:J1} and \eqref{eq:J2}.\\

Hereafter we fix a constant $T>0$ which is the final time of the considered process. Given the spaces $\Omega_{T}=C([0,T];\Real^n)$ and $\Omega_{T}^N=C([0,T];\Real^{nN})$, we denote by $\mathbb{P}_0$ the law of the solution to the SDE \eqref{eq:SDEmain} at the optimal control $\alpha=\frac{\nabla \rho}{\rho}$ and with initial condition $\rho_0$. Moreover, we denote by $\mathbb P_{0,N}$ the law of  the system of $N$ interacting diffusions \eqref{eq:SDEmain2} at the optimal control $A_N=\frac{\nabla \rho_{0,N}}{\rho_{0,N}}$ and with initial condition $\rho_{0,N}$. We  write $\mathbb{P}_{0,N}^{(k)}$ (for $N \geq k$) for the probability measure obtained by projecting $\mathbb{P}_{0,N}$ onto $\Omega_T^k$ (the path space of the first $k$ particles). \\

We introduce the notion of relative entropy.

\begin{definition}
If $\mathbb{P}$ and $\mathbb{Q}$ are two probability laws on the same probability space $\Omega$, such that $\mathbb{P}$ is absolutely continuous with respect to $\mathbb{Q}$, the relative entropy (Kullback–Leibler divergence) between $\mathbb{P}$ and $\mathbb{Q}$ is defined by 
$$\mathcal{H}_{\Omega}(\mathbb{P}|\mathbb{Q})=\int_{\Omega}{\log\left(\frac{d\mathbb{P}}{d\mathbb{Q}}(\omega)\right)\mathbb{P}(d\omega)},$$
and it is defined to be infinity if $\mathbb{P}$ is not absolutely continuous with respect to $\mathbb{Q}.$
\end{definition}

The following result establishes a strong form of entropy chaos for the probability laws associated with the $N$-particles optimal control problem.

\begin{theorem}\label{theorem_main1}
Under hypotheses $\mathcal{V}$, C$\mathcal{V}$ and Q$V$ we have that for all $ k\in \mathbb{N}$
\begin{equation}
\lim_{N\uparrow +\infty}\mathcal{H}_{\Omega_T^k}({\mathbb P}^{(k)}_{0,N}|{\mathbb{P}}^{\otimes k}_0)=0.
\end{equation}
\end{theorem} 

The proof of Theorem \ref{theorem_main1} will be given in Section \ref{subsection_proof2} below. First we would like to point out some consequences of Theorem \ref{theorem_main1}.

\subsection{Some consequences of Theorem \ref{theorem_main1}}

In this section we want to discuss some consequences of Theorem \ref{theorem_main1} in particular concerning the convergence of the probability measures ${\mathbb P}^{(k)}_{0,N}$ to ${\mathbb{P}}^{\otimes k}_0$ with respect to the total variation metric and the Wasserstein metric on the space of (Borel) probability measures $\mathcal{P}(\Omega)$ on $\Omega$.

\begin{definition}\label{definition_wesserstein}
Let $\Omega$ be a Polish space with metric $d_{\Omega}$, and let us denote by $\mathcal{P}(\Omega)$ the set of probability measures on $\Omega$. If $\mathbb{P},\mathbb{Q} \in \mathcal{P}(\Omega)$ the total variation distance between $\mathbb{P}$ and $\mathbb{Q}$ is the following non-negative real number 
$$d_{TV,\Omega}(\mathbb{P},\mathbb{Q})=\sup_{A\in\mathcal{B}(\Omega)}|\mathbb{P}(A)-\mathbb{Q}(A)|,$$
where $\mathcal{B}(\Omega)$ denotes the $\sigma$-algebra of Borel subsets of $\Omega$.\\
If $\mathbb{P},\mathbb{Q} \in \mathcal{P}(\Omega)$ and $p\geq 1$ we call the $p$-th Wasserstein distance between $\mathbb{P}$ and $\mathbb{Q}$, $W_{p,\Omega}$,  the positive number
$$W_{p,\Omega}(\mathbb{P},\mathbb{Q} )=\left( \inf \left\{\int_{\Omega \times \Omega}{d_E(x,y)^p\pi(x,y)}, \text{ where }\pi\in \mathcal{P}(\Omega \times \Omega) \text{ such that } P_{1,*}\pi=\mathbb{P},P_{2,*}\pi=\mathbb{Q} \right\}  \right)^{\frac{1}{p}},$$
where  $P_1:\Omega\times \Omega \rightarrow \Omega$ and $P_2:\Omega\times \Omega \rightarrow \Omega$ are respectively the projections on the first and second component of $\Omega \times \Omega$.
\end{definition}

\begin{corollary}\label{corollary_TV}
Under hypotheses $\mathcal{V}$, C$\mathcal{V}$ and Q$V$ we have that for all $ k\in \mathbb{N}$
\begin{equation}
\lim_{N\uparrow +\infty}d_{TV,\Omega_{T}^k}({\mathbb P}^{(k)}_{0,N},{\mathbb{P}}^{\otimes k}_0)=0.
\end{equation}
\end{corollary}
\begin{proof}
The thesis is a consequence of the well known Csiszar-Kullback inequality
(\cite{Csiszar},\cite{Kullback}), which is valid in arbitrary Polish spaces, yielding
\begin{equation}\label{CKinequality}
d_{TV,\Omega}({\mathbb P},\mathbb{Q}) \le \sqrt{2 \mathcal H_{\Omega} ({\mathbb P}|\mathbb{Q}) }.
\end{equation}
\end{proof}

Hereafter if $\Omega$ is a separable Banach space and $\mathbb{P} \in \mathcal{P}(E)$ we write 
\begin{equation}\label{eq:moment}
M_{k,\Omega}(\mathbb{P})=\int_{\Omega} ||x||_{\Omega}^k\mathbb{P}(dx).
\end{equation}

\begin{lemma}\label{lemma_Wesserstein}
Under Hypotheses $\mathcal{V}$, C$\mathcal{V}$ and Q$V$, there is $k>2$ (depending on the constant $\epsilon>0$ in Hypothesis Q$V$) such that for any $h\in \mathbb{N}$
$$\sup\left\{\left\{M_{k,\Omega^h}(\mathbb{P}_{0,N}^{(h)})|N\in \mathbb{N}, N\geq h\right\},M_{k,\Omega^h}(\mathbb{P}_{0}^{\otimes h})\right\}<+\infty$$
\end{lemma}
\begin{proof}
We provide the proof for $h=1$ and $\mathbb{P}_{0,N}$, being the general case a straightforward generalization. Consider $\epsilon_1>0$ which we fix later in a suitable way.  By It\^o formula (applied to $|X^1_t|^2$), Jensen inequality, Doob inequality and Young inequality we have
\begin{eqnarray}
\mathbb{E}_{\mathbb{P}_{0,N}}\left[\left(\sup_{t\in[0,T]}|X_t^1|^2\right)^{1+\epsilon_1}\right]&\leq&2^{1+\epsilon_1}\mathbb{E}_{\mathbb{P}_{0,N}}\left[\sup_{t\in[0,T]}\left||X_t^1|^2-2\int_0^t{A_N^1(X_s^N) \cdot X^1_s ds}-2t\right|^{1+\epsilon_1}\right]+\nonumber \\
&&+2^{2+\epsilon_1}\mathbb{E}_{\mathbb{P}_{0,N}}\left[\left(\int_0^T{|A_N^1(X_s)|| X_s^1| ds}\right)^{1+\epsilon_1}\right]+2^{2+\epsilon_1}T^{1+\epsilon_1}\nonumber\\
&\leq&2^{2+2\epsilon_1}\mathbb{E}_{\mathbb{P}_{0,N}}[|X_0^1|^{2+2\epsilon_1}]+2^{3+2\epsilon_1}\mathbb{E}\left[\sup_{t\in[0,T]}\left|\int_0^tX_t^1\cdot dW^1_t\right|^{1+\epsilon_1}\right]+\nonumber\\
&&+2^{2+\epsilon_1}T^{\epsilon_1}\mathbb{E}_{\mathbb{P}_{0,N}}\left[\int_0^T{|A_N^1(X_s)|^{1+\epsilon_1}| X_s^1|^{1+\epsilon_1} ds}\right]+2^{2+\epsilon_1}T^{1+\epsilon_1}\nonumber
\end{eqnarray}
\begin{eqnarray}
\phantom{\mathbb{E}_{\mathbb{P}_{0,N}}\left[\left(\sup_{t\in[0,T]}|X_t^1|^2\right)^{1+\epsilon_1}\right]}&\leq&2^{2+2\epsilon_1}\mathbb{E}_{\mathbb{P}_{0,N}}[|X_0^1|^{2+2\epsilon_1}]+\frac{2^{3+2\epsilon_1}(1+\epsilon_1)}{\epsilon_1}\mathbb{E}_{\mathbb{P}_{0,N}}\left[\left|\int_0^T
X_t^1\cdot dW^1_t\right|^{1+\epsilon_1}\right]+\nonumber\\
&&+2^{1+\epsilon_1}(1+\epsilon_1)T^{\epsilon_1}\mathbb{E}_{\mathbb{P}_{0,N}}\left[\int_0^T|A^1_N(X_t)|^2dt\right]+\frac{2^{3+\epsilon_1}T^{\epsilon_1}}{1-\epsilon_1}\times\nonumber\\
&&\times \mathbb{E}_{\mathbb{P}_{0,N}}\left[\int_0^T|X^1_t|^{\frac{2(1+\epsilon_1)}{(1-\epsilon_1)}}dt
\right]+2^{2+\epsilon_1}T^{1+\epsilon_1}.\label{eq:Wasserstein10}
\end{eqnarray}
From inequality \eqref{eq:Wasserstein10} using Jensen inequality, It\^o isometry and the stationarity of the distribution of $X_t$ we get
\begin{eqnarray}
\mathbb{E}_{\mathbb{P}_{0,N}}\left[\left(\sup_{t\in[0,T]}|X_t^1|^2\right)^{1+\epsilon_1}\right]&\leq&2^{2+2\epsilon_1}\mathbb{E}_{\mathbb{P}_{0,N}}[|X_0^1|^{2+2\epsilon_1}]+\frac{2^{3+2\epsilon_1}(1+\epsilon_1)T^{\frac{1+\epsilon_1}{2}}}{\epsilon_1}
\left(\mathbb{E}_{\mathbb{P}_{0,N}}\left[|X_0^1|^2\right]\right)^{\frac{1+\epsilon_1}{2
}}+\nonumber\\
&&+2^{1+\epsilon_1}(1+\epsilon_1)T^{1+\epsilon_1}\mathbb{E}_{\mathbb{P}_{0,N}}\left[|A^1_N(X_0)|^2\right]+\frac{2^{3+\epsilon_1}T^{1+\epsilon_1}}{1-\epsilon_1}
\mathbb{E}\left[|X^1_0|^{\frac{2(1+\epsilon_1)}{(1-\epsilon_1)}} \right]+\nonumber\\
&&+2^{2+\epsilon_1}T^{1+\epsilon_1}\label{eq:moments1}
\end{eqnarray}
By choosing $\epsilon_1$ small enough, i.e. such that $2\epsilon_1<\epsilon$ and $\frac{2(1+\epsilon_1)}{1-\epsilon_1}<2+\epsilon$ (recolling that $\epsilon>0$ is the constant in Hypothesis Q$V$) inequality \eqref{eq:moments1} implies 
\begin{multline}
\mathbb{E}_{\mathbb{P}_{0,N}}\left[\left(\sup_{t\in[0,T]}|X_t^1|^2\right)^{1+\epsilon_1}\right]\leq P_1(T,\epsilon_1)\left(\mathbb{E}_{\mathbb{P}_{0,N}}[|A_N^1(X_t)|^2+|X_t^1|^{2+\epsilon}] \right)+P_2(T,\epsilon_1) \leq \\
\leq Q_1(T,\epsilon_1)\mathfrak{J}_N+Q_2(T,\epsilon_1)
\end{multline}
where $ P_1(T,\epsilon_1), P_2(T,\epsilon_1),Q_2(T,\epsilon_1),Q_2(T,\epsilon_1)\in \mathbb{R}_+$ are suitable functions of $T$ and $\epsilon_1>0$. Since, by Theorem \ref{theorem_VE}, $\sup_{N\in \mathbb{N}}\mathfrak{J}_N<+\infty$, the thesis is proved by choosing $k=2+2\epsilon_1$.
\end{proof}

\begin{corollary}\label{corollary_Wesserstein}
Under hypotheses $\mathcal{V}$, C$\mathcal{V}$ and Q$V$ we have that, for all $1 \leq p \leq 2$ and $k\in \mathbb{N}$,
\begin{equation}\label{eq:Wesserstein}
\lim_{N\uparrow +\infty}W_{p,\Omega_{T}^k}({\mathbb P}^{(k)}_{0,N},{\mathbb{P}}^{\otimes k}_0)=0.
\end{equation}
\end{corollary}
\begin{proof}
The proof is based on the following inequality (see, e.g., \cite{Villani} Theorem 6.15)
\begin{equation}\label{eq:Wesserstein2}
W_{p,\Omega}(\mathbb{P},\mathbb{Q})\leq 2^{1-\frac{1}{p}} \left(\int_{\Omega}{\|x\|_{\Omega}^{p}|\mathbb{P}-\mathbb{Q}|(dx)}\right)^{\frac{1}{p}},
\end{equation}
which holds for any separable Banach space $\Omega$ and $1 \leq p$. Indeed, applying H\"older inequality and the fact that $|\mathbb{P}-\mathbb{Q}|\leq \mathbb{P}+\mathbb{Q}$ to relation \eqref{eq:Wesserstein2}, we get
\begin{equation}\label{eq:Wesserstein3}
W_{p,\Omega}(\mathbb{P},\mathbb{Q}) \leq 2^{1-\frac{1}{p}} \left(M_{k,\Omega}(\mathbb{P})+M_{k,\Omega}(\mathbb{Q})\right)^{\frac{1}{k}}\left( d_{TV,\Omega}(\mathbb{P},\mathbb{Q})\right)^{\frac{k}{(k-p)p}},
\end{equation}
for any $p< k$.  Thus, the thesis follows by applying Corollary \ref{lemma_Wesserstein} and Lemma \ref{lemma_Wesserstein} to inequality \eqref{eq:Wesserstein3} with $\Omega=\Omega_T^k$.
\end{proof}

\begin{remark}
If we have that $V(x)\geq d_1 |x|^{(2 P-2)+\epsilon}-d_2$ for some constants $d_1,\epsilon>0$, $d_2 \in \mathbb{R}$ and $P\geq 2$, then, using the techniques of the proofs of Lemma \ref{lemma_Wesserstein} and Corollary \ref{corollary_Wesserstein}, it is easy to prove convergence \eqref{eq:Wesserstein} for any $1 \leq p \leq P$.
\end{remark}

\subsection{Proof of Theorem \ref{theorem_main1}}\label{subsection_proof2}

Before giving the proof we prove some preliminary lemmas.

\begin{lemma}\label{lemma_entropy}
Under hypotheses $\mathcal{V}$,  for any $N \geq 2$ and $s\geq 0$ we have
\begin{multline}\label{equation_drift} 
\frac{1}{N}{\mathbb E}_{\mathbb P_{0,N}}[  |A^1_N(X_s) -  \alpha(X^1_s) |^2 ]=\int_{\mathbb{R}^{nN}}{\frac{|\nabla_1\phi_{0,N}(x)|^2}{2}dx}-\mu_{0}+\\
+\int_{\mathbb{R}^{nN}}{2\left(\mathcal{V}(x^1,\rho_0)-\int_{\mathbb{R}^n}{\partial_{\mu}\mathcal{V}(y,x^1,\rho_0)\rho(y)dy}\right)\phi_{0,N}^2(x)dx},
\end{multline}
where $\mu_0$ is defined in \eqref{eq:mu0}.
\end{lemma}
\begin{proof}
By a simple computation and recalling that, by Lemma \ref{theorem_minimizer1} and Remark \ref{remark_minimizer1}, $\phi_{0}$ is strictly positive and $C^2$ we have
$$\frac{1}{N}{\mathbb E}_{\mathbb P_{0,N}}[  |A^1_N(X_s) -  \alpha(X^1_s) |^2 ]=\frac{1}{2}\int_{\mathbb R^{nN}}{\left|\nabla_1\left( \frac{\phi_{0,N}}{\phi_{0}} \right) \right|^2 \phi_{0}^2 dx}.$$
We now prove that $\int_{\mathbb R^{nN}}{\left|\nabla_1\left( \frac{\phi_{0,N}}{\phi_{0}} \right) \right|^2 \phi_{0}^2 dx}$ is finite and equal to the right hand side of equation \eqref{equation_drift}. Let us denote by $\Psi_{R,N}$  the ground state of equation \eqref{eq:variational2} restricted to the ball $B_R$, having radius $R$ and centered in $0$, with Dirichlet boundary condition (i.e. $\Psi_{R,N}$ is the solution to equation \eqref{eq:variational2} for the minimal constant $\mu_N$). Integrating by parts, and exploiting that $\Psi_{N,R}|_{\partial B_R}=0$ and equation \eqref{eq:variational1} we obtain 
\begin{multline}
\frac{1}{2}\int_{B_R}{\left|\nabla_1\left( \frac{\Psi_{N,R}}{\phi_{0}} \right) \right|^2 \phi_{0}^2 dx}
= \int_{B_R}{\left(\frac{\left|\nabla_1\Psi_{N,R}\right|^2}{2}-\frac{1}{2}\nabla_1\left(\frac{|\Psi_{N,R}|^2}{\phi_0}\right)\cdot \nabla_1\phi_{0} \right)dx}\\
=\int_{B_R}{\frac{|\nabla\Psi_{N,R}|^2}{2}dx}-\mu_0+\int_{B_R}{2\left(\mathcal{V}(x^1,\rho_0)-\int_{\mathbb{R}^n}{\partial_{\mu}\mathcal{V}(y,x^1,\rho_0)\rho_0(y)dy}\right)|\Psi_{N,R}|^2dx}
\end{multline}
Using the fact that $R\uparrow \infty$, and the density of  regular functions with compact support is  in $H^1(\mathbb{R}^{nN})$, we have that $\mathcal{E}_N(|\Psi_{N,R}|^2) \rightarrow \mathcal{E}_{N}(|\phi_{0,N}|^2)$. By exploiting a reasoning similar to the one used in the proof of Theorem \ref{theorem_VE}, we  prove that $|\Psi_{N,R}|^2$ converges weakly to $|\phi_{0,N}|^2$ (in $L^1(V(x)dx)$) and that  $\int_{B_R}{\frac{|\nabla\Psi_{N,R}|^2}{2}dx} \rightarrow \int_{\mathbb{R}^{nN}}{\frac{|\nabla\phi_{0,N}|^2}{2}dx}$. This concludes the proof of the Lemma \ref{lemma_entropy}.
\end{proof}

\begin{remark}\label{remark_entropy}
An important consequence of Lemma \ref{lemma_entropy} and Theorem \ref{theorem_VE}  is that, as $N \rightarrow \infty$, 
$$\frac{1}{N}{\mathbb E}_{\mathbb P_{0,N}}[  | A^1_N(X_s) -  \alpha(X^1_s)|^2 ] \rightarrow 0,$$
for any $s \geq 0$.
\end{remark}

We introduce the following notation: when $\mathbb{P}$ and $\mathbb{Q}$ are defined on $\Omega_T^N$ we denote by $\overline{\mathcal{H}}_{\Omega_T^N}$ the normalized relative entropy given, for all $N\in \mathbb{N}$, by 
\begin{equation}\label{eq:normalizedentropy}
\overline{\mathcal{H}}_{\Omega_T^N}(\mathbb{P}|\mathbb{Q}):=\frac{1}{N} \mathcal{H}_{\Omega_T^N}(\mathbb{P}|\mathbb{Q}).
\end{equation}
The following lemma provides the expression of the normalized relative entropy in our framework.

\begin{lemma}\label{lemma2}
Under the Hypotheses $\mathcal{V}$
we have that 

 \begin{equation}
 \overline{\mathcal H}_{\Omega^N_T} (\mathbb P_{0,N}|\mathbb P^{\otimes N}_0) =-\frac{1}{N} H_N(\rho_{0,N}|\rho^{\otimes N}_0)+\frac 1 2 {\mathbb E}_{\mathbb P_{0,N}}\left[\int_0^{T}  | A^1_N(X_s) - \alpha(X^1_s)|^2 d s\right].
 \end{equation}

\end{lemma}

\begin{proof}
The proof runs in a way similar to the one performed for the Gross-Pitaevskii scaling limit in \cite{MU}. Here are the details.
As a consequence of Lemma \ref{lemma_entropy} we have that $ \forall T>0$  
\begin{equation}\label {eni}
{\mathbb E}_{{\mathbb P}_{0,N}} \int_0^T | A^1_N(X_s) |  ^2 ds  <  + \infty
\end{equation}
\begin{equation}\label{enii}
{\mathbb E}_{{\mathbb P}_{0,N}} \int_0^T| \alpha(X^i_s)  | ^2 ds   < +\infty .
\end{equation}
The inequalities \eqref{eni} and \eqref{enii} are  \textit{finite entropy conditions} (see, e.g.
\cite{Follmer}) which imply that for all $ T > 0$
$$\mathbb{P}_{0,N} \ll dx\otimes W^{\otimes N}, \quad \mathbb{P}^{\otimes N}_{0}\ll dx\otimes W^{\otimes N}  $$
(where $\ll$ stands for absolute continuity and $W$ is the law of Brownian motion on $\Omega_T$ and $dx$ is the Lebesgue measure on the initial condition $X_0 \in \mathbb{R}^{nN}$).
 By applying Girsanov's theorem, we obtain in a standard way that, for all $T>0$, the Radon-Nikodym derivative restricted to the time $T$ is given by
\begin{multline}\label{derivative}
 \frac {d\mathbb{P}_{0,N}}{d\mathbb{P}^{\otimes N}_{0}} =\exp \left\{\log\left(\frac{\rho_0^{\otimes N}(X_0)}{\rho_{0,N}(X_0)}\right)\right. \\
\left. -\sum_{i=1}^N\int_0^{T}  ( A^i_N(X_s)-
 \alpha(X^i_s))\cdot
 d  W_s+\frac{1}{2} \int_0^ {T}  | A^i_N(X_s) - \alpha(X^i_s)|^2ds
\right\}.
 \end{multline}
The relative entropy  reads
 \begin{multline} \label{eq: Entropy}
{\mathcal H}_{\Omega_T^N}(\mathbb{P}_{0,N}|\mathbb{P}^{\otimes N}_{0}) 
 =: {\mathbb E}_{\mathbb{P}_{0,N}}\left[
 \log\left( \frac {d\mathbb{P}_{0,N}}{d\mathbb{P}^{\otimes N}_{0}}\right)\right]=\\
=-\mathbb{E}_{\mathbb{P}_{0,N}}\left[\log\left(\frac{\rho_{0,N}(X_0)}{\rho_0^{\otimes N}(X_0)} \right)\right]+ \sum_{i=1}^N\frac {1}{2}{\mathbb E}_{\mathbb{P}_{0,N}}\int_0^{T}  | A^i_N(X_s) - \alpha(X^i_s)|^2 ds
 \end{multline}
Since under $\mathbb P_{0,N}$ the $nN$-dimensional process $X$ is a solution of \eqref {eq:SDEmain2}
 with invariant probability density $ \rho_{0,N}$ , we get, recalling also \eqref{eni} and \eqref{enii},
and by using  the symmetry of $A^i_N(x) $ and $\rho_{0,N}$ with respect to coordinates permutations (see Remark \ref{remark_symmetry}) 

$$  \mathcal H_{\Omega_T^N}(\mathbb{P}_{0,N}|\mathbb{P}^{\otimes N}_{0})
 =-\int_{\mathbb{R}^{nN}}{\log\left(\frac{\rho_{0,N}(x)}{\rho_0^{\otimes N}(x)}\right)\rho_{0,N}(x)dx}+
 \frac 1 2 N T\int _{\R^{nN}} | A^1_N(x) -  \alpha(x^1)|^2 \rho_{0,N}(x) dx
$$ 
By definition of normalized relative entropy this concludes the proof of Lemma \ref{lemma2}.
\end{proof}

We recall an interesting property of the relative entropy in the case in which the second measure is a product measure.

\begin{lemma}\label{lemma3} We consider $\Omega=\Omega_1\times \Omega_2$, where $\Omega_1$ and $\Omega_2$ are
Polish spaces. Let ${\mathbb P}$ be a measure on $\Omega$ and ${\mathbb
Q}_1$ and ${\mathbb Q}_2$ probability measures on $\Omega_1$ and $\Omega_2$ respectively. We
denote by $\mathbb Q={\mathbb Q}_1\otimes {\mathbb Q}_2$ the
product measure on $\Omega$ of the measures ${\mathbb Q}_1$ and
${\mathbb Q}_2$ and we suppose that ${\mathbb P} \ll {\mathbb Q}$.
Then we have
\begin{equation}
{\mathcal H}_{\Omega}({\mathbb P}|{\mathbb Q})\geq {\mathcal H}_{\Omega_1}({\mathbb
P}_1|{\mathbb Q}_1) + {\mathcal H}_{\Omega_2}({\mathbb P}_2|{\mathbb Q}_2),
\end{equation}
\noindent where ${\mathbb P}_1$ and ${\mathbb P}_2$ are the
marginal probabilities of ${\mathbb P}$.
\end{lemma}
\begin{proof}
The proof can be found in Lemma 5.1 of \cite{DeVU}.
\end{proof}

\begin{proof}[Proof of Theorem \ref{theorem_main1}]
We prove the statement by induction on $k$. \\

Take first $k=1$. By applying Lemma \ref{lemma3} we have, for $N \geq 2$, 
\begin{equation}
{\mathcal H}_{\Omega^N_T}({\mathbb P}_{0,N}|\mathbb{P}^{\otimes N}_0)\geq {\mathcal H}_{\Omega_T}({\mathbb
P}^{(1)}_{0,N}|{\mathbb P}_{0}) + {\mathcal H}_{\Omega_T^{N-1}}({\mathbb P}^{(N-1)}_{0,N}|{\mathbb P}^{\otimes (N-1)}_0),
\end{equation}
and by repeating the same procedure we obtain
\begin{equation}\label{onerelativeentropy}
\mathcal H_{\Omega_T}(\mathbb P^{(1)}_{0,N}|\mathbb P_0) \le \overline{\mathcal H}_{\Omega_T^N} (\mathbb P_{0,N}, \mathbb P^{\otimes N}_0),
\end{equation}
where $\overline{\mathcal{H}}$ is the normalized entropy introduced in \eqref{eq:normalizedentropy}. Using Theorem \ref{theorem_HH1}, Lemma \ref{lemma2} and Remark \ref{remark_entropy} we have proved the thesis for $k=1$.\\
For generic $k$, let us write $N=kN_k+r_k$, with $N_k\in \mathbb{N}, r_k=0,...,k-1$, and suppose that the statement is true for any $r_k < k$. By  Lemma \ref{lemma3} we have 
\begin{equation}
\mathcal{H}_{\Omega_T^N}(\mathbb{P}_{0,N}|\mathbb{P}^{\otimes N}_0)\geq N_k {\mathcal H}_{\Omega_T^k}({\mathbb P}^{(k)}_{0,N}|{\mathbb P}^{\otimes k}_0) + {\mathcal H}_{\Omega_T^{r_k}}({\mathbb P}^{(r_k)}_{0,N}|{\mathbb P}^{\otimes r_k}_0),
\end{equation}
which implies:
\begin{multline}
{\mathcal H}_{\Omega_T^k}({\mathbb
P}^{(k)}_{0,N}|{\mathbb P}^{\otimes k}_0) \leq \frac{1}{N_k} \left\{{\mathcal H}_{\Omega_T^N}({\mathbb P}_{0,N}|\mathbb{P}^{\otimes N}_0)+
{\mathcal H}_{\Omega^{r_k}_T}({\mathbb P}^{(r_k)}_{0,N}|{\mathbb P}_0^{\otimes r_k})\right\}\\
\leq \frac{N}{N_k}\left\{\overline{\mathcal H}_{\Omega_T^N}({\mathbb P}_{0,N}|\mathbb{P}^{\otimes N}_0)\right\}+
\frac{1}{N_k}{\mathcal H}_{\Omega^{r_k}_T}({\mathbb P}^{(r_k)}_{0,N}|\mathbb{P}_0^{\otimes r_k})
\end{multline}
Since  when $N \uparrow \infty$ we have $\frac{N}{N_k}\rightarrow k$ and, by Theorem \ref{theorem_HH1}, Lemma \ref{lemma2} and Remark \ref{remark_entropy}, \\ $\lim_{N\uparrow +\infty} \overline{\mathcal H}_{\Omega_T^N} (\mathbb P_{0,N}, \mathbb P^{\otimes N}_0)\rightarrow  0$,  we obtain the desired result from the induction hypothesis 
$$\lim_{N\uparrow +\infty} {\mathcal H}_{\Omega_T^{r_k}}({\mathbb P}^{(r_k)}_{N,0}|{\mathbb P}^{\otimes r_k}_0) \rightarrow  0,$$
since $r_k<k$.
\end{proof}

\section{The case of the Dirac delta potential}\label{section_physics}

In this section we propose to the reader a potential $\mathcal{V}$ of the following form 
\begin{equation}\label{eq:nl1}
\mathcal{V}_{\delta}(x,\mu)=V_0(x)+g\delta_x * \mu
\end{equation}
where $V_0$ is a regular positive function growing at infinity, $\delta_x$ is the Dirac delta centered at $x\in \mathbb{R}^n$, $g\in \mathbb{R}_+$ is a strictly positive constant and $*$ stands for convolution. The potential $\mathcal{V}_{\delta}$ does not satisfies the regularity Hypothesis $\mathcal{V}$\emph{i} and $\mathcal{V}$\emph{iii}. On the other hand it satisfies Hypothesis $\mathcal{V}$\emph{ii} and C$\mathcal{V}$, and (when the G\^ateaux derivative is well defined) we have $\partial_{\mu}(\tilde{\mathcal{V}}_{\delta})=2\delta_{x-y}$, where $\tilde{\mathcal{V}}_{\delta}$ is defined as in \eqref{functionalofmu},  which is a  positive definite distribution.  A similar singular problem has been considered in the case of mean-field games in \cite{flandoli2021n}.\\

Here we do not consider the problem of proving that the optimal control ergodic problem has a unique optimal control (i.e. we do not prove here the equivalent of Theorem \ref{theorem_optimalcontrol1} for the potential  \eqref{eq:nl1}). We suppose here that there exists a family $\tilde{\mathcal{C}}_{V_0} \subset C^1(\mathbb{R}^n,\mathbb{R}^n)$ of vector fields $\alpha$ (in general we expect that it can depend on the trapping potential $V_0$ in \eqref{eq:nl1}, for example we can take $\tilde{\mathcal{C}}_{V_0}=\mathcal{C}_{\rho_0}$ as defined in \eqref{eq:Crho}) such that 
\begin{multline}\label{eq:g2}
\inf_{\alpha\in \tilde{\mathcal{C}}_{V_0}}\left(
\limsup_{T \rightarrow +\infty}\frac{1}{T} \left(\int_0^T\mathbb{E}_{x_0}\left[\frac{|\alpha(X_t)|^2}{2}+V_0(X_t)+g\rho_{x_0,\alpha,t}(X_t)^2 \right]dt \right)\right)=\\
=\mathbb{E}_{X_0\sim \rho_0(x)dx}\left[\frac{|\nabla \rho_0(X_t)|^2}{2\rho_0^2(X_t)}+V_0(X_t)+g(\rho_{0}(X_t))^2 \right],
\end{multline}
where $\rho_{x_0,\alpha,t}$ is the probability density of the law of the solution to the SDE \eqref{eq:SDEmain} starting at $x_0\in \mathbb{R}^n$ evaluated at time $t$, and $\rho_0$ is the density of the probability distribution minimizing the functional 
\begin{equation}\label{eq:nl2}
\mathcal{E}_{\delta}(\rho)=\mathcal{E}_{K}(\rho)+\mathcal{E}_{\delta,P}(\rho)=\int_{\mathbb{R}^n}{\left(\frac{|\nabla \rho(x)|^2}{\rho(x)}+V_0(x)\rho(x)+g (\rho(x))^2\right)dx}.
\end{equation}
In other words we suppose that in the set $\mathcal{C}$ (introduced in Section 3.2) the optimal control for the problem \eqref{eq:SDEmain} with cost functional \eqref{eq:J1} and potential $\mathcal{V}_{\delta}$ (see \cite{AU} for an alternative derivation of a stochastic process associated with the above cost functional) exists and it is given by $\alpha=\frac{\nabla \rho_0}{\rho_0}$. 
What we want to consider here is an $N$-particle problem converging to the solution of the optimal control ergodic problem just described (namely we are looking for an analogous of Theorem \ref{theorem_main1}). \\

Obviously, since $\mathcal{V}_{\delta}$ is not well defined for measures $\mu$ that are not absolutely continuous measures, we consider here an approximating potential of the form 
$$\mathcal{V}_{\delta,N}(x,\mu)=V_0(x)+ \int_{\mathbb{R}^n}v_N(x-y) \mu(dy),$$
where $v_N:\mathbb{R}^n \rightarrow \mathbb{R}$ is a sequence of positive functions converging in the sense of distributions to a Dirac delta $\delta_0$ when $N \rightarrow \infty$. Let us choose a specific sequence of the following form 
\begin{equation}\label{eq:beta}
v_N(x)=N^{n\beta} v_0\left(N^{\beta} x\right), \ x \in \mathbb{R}^n 
\end{equation}
for $\beta>0$, where $v_0$ is a positive smooth radially symmetric function with compact support. We take the $N$-particles approximation having the control $A(x_1,...,x_N)$ given by the logarithm derivative of $\rho_{0,N}$ that is the minimal probability density of the energy functional $\mathcal{E}_{\delta}$ associated with $\mathcal{V}_{\delta,N}$, namely 
\begin{multline*}
\mathcal{E}_{\delta,N}(\rho)=\mathcal{E}_{K,N}(\rho)+\mathcal{E}_{\delta,P,N}(\rho)=\\
=\frac{1}{N}\sum_{i=1}^N \left(\int_{\mathbb{R}^{Nn}}{\left(\frac{|\nabla_i \rho|^2}{\rho}+V_0(x_i)\rho\right)dx}+\frac{1}{N-1}\sum_{j=1,...,N,j \not =i}\int_{\mathbb{R}^{Nn}}{v_N(x_i-x_j)\rho dx} \right).
\end{multline*}

In the rest of the paper we show how the results on Bose-Einstein condensation (mainly for $n=3$, see, e.g., \cite{Lewin2,Lewin,Lieb2,LiebBook,Lieb,Nam,Rougerie}) can be used to study the convergence of the $N$-particles approximation of the control problem with potential \eqref{eq:nl1}. For this reason hereafter we shall limit our discussion to the case $n=3$. 

\subsection{Intermediate scaling limit}

The case $0 < \beta < 1$, where $\beta$ is the parameter used in the rescaling \eqref{eq:beta}, which is known as intermediate scaling limit, is very similar to the regular case that we treated in the first part of the paper. Indeed in this case we can prove the following theorem.

\begin{theorem}\label{theorem_delta_int}
Under the previous hypotheses and notations, if $0 < \beta < 1$  we have the $N \rightarrow +\infty$ convergence statements $\mathcal{E}_{\delta,N}(\rho_{0,N}) \rightarrow \mathcal{E}_{\delta}(\rho_{0})$, $\mathcal{E}_{\delta,P,N}(\rho_{0,N}) \rightarrow \mathcal{E}_{\delta,P}(\rho_{0})$ and $\rho_{0,N}^{(1)} \rightarrow \rho_0$ (where the later convergence is in the weak $L^1$ sense) with the constant $g=\int_{\mathbb{R}^3}{v_0(x)dx}$ (where $g\in \mathbb{R}_+$ is the constant appearing in equation \eqref{eq:nl1} and \eqref{eq:g2}).
\end{theorem}
\begin{proof}
The proof of the theorem can be found in  \cite{Lewin} for $0 \leq \beta <\frac{1}{3}$ (for any $n$ and a more general class of potentials $v_0$ than the one considered here) and in \cite{Strong} for $0 \leq \beta < 1$  (for $n=3$ and a positive-definite interaction potential $v_0$).
\end{proof}

Theorem \ref{theorem_delta_int} is the analogue of Theorem \ref{theorem_VE} in this context and it proves that $\mathcal{E}_{\delta,N}$ and $\mathcal{E}_{\delta}$ satisfy the thesis of Theorem \ref{theorem_VE}. Thanks to Theorem \ref{theorem_delta_int} we can repeat the reasoning performed in Section \ref{section_convergence}, obtaining:

\begin{theorem}
Under the previous hypotheses and notations, if $0 < \beta < 1$ we have that the law $\mathbb{P}^{(k)}_{0,N}$ of the first $k$ particles satisfying the system \eqref{eq:SDEmain2}, with $\mathcal{V}$ replaced by $\mathcal{V}_{\delta}$, converges in total variation on the path space $C([0,T],\mathbb{R}^{3k})$ to $\mathbb{P}_0^{\otimes k}$ (where $\mathbb{P}_0$ is the law on $C([0,T],\mathbb{R}^3)$ of the system \eqref{eq:SDEmain} associated with \eqref{eq:nl1}).
\end{theorem}
\begin{proof}
The proof can be found in \cite{Strong}.
\end{proof}

\subsection{Gross-Pitaevskii scaling limit}

The case $\beta=1$ is completely different with respect to the previous ones. The main difference between the cases $0<\beta<1$ and $\beta=1$ is that in this latter case the value function convergence result \eqref{LimE2} does not hold.

\begin{theorem}\label{theorem_delta1}
Under the previous hypotheses and notations, if $ \beta = 1$  we have that $\mathcal{E}_{\delta,N}(\rho_{0,N}) \rightarrow \mathcal{E}_{\delta}(\rho_{0})$ and $\rho_{0,N}^{(1)} \rightarrow \rho_0$ (where the latter convergence is in the weak sense in $L^1$)
for $g=4 \pi a$ (where $g\in \mathbb{R}_+$ is the constant appearing in equation \eqref{eq:nl1} and \eqref{eq:g2}, and $a>0$ is the scattering length of the interaction potential $v_0$ (a sort of effective range of the interaction potential, for details see \cite{LiebBook})). Furthermore putting $\hat{s}=\frac{1}{g}\int_{\mathbb{R}^3}{\frac{|\nabla\rho_0|^2}{\rho_0}dx}\in (0,1)$ we have, as $N \rightarrow +\infty$:
$$\mathcal{E}_{K,N}(\rho_{0,N})\rightarrow\mathcal{E}_{\delta,K}(\rho_0)+g\hat{s}\int_{\mathbb{R}^3}{\rho_0^2(x)dx}.$$
\end{theorem}
\begin{proof}
The proof of the first part of the theorem is a, by this time, well-known relevant result proven in \cite{Lieb2,Lieb,Nam}. The second part is proven in \cite{Lieb}.
\end{proof}

In this case we cannot repeat the reasoning of Section \ref{section_convergence} since we are not able to prove that the relative entropy $\mathcal{H}(\mathbb{P}^{(k)}_{0,N}|\mathbb{P}_0^{\otimes k})$ converges to $0$ (in fact we do not know whether the entropy converges to $0$ or to another value). On the other hand it is possible to prove a weaker result for $\beta=1$ (see \cite{MU} for a different kind of convergence and \cite{Ugolini} for a transition to chaos result).

\begin{theorem}
Under the previous hypotheses and notations, if $\beta = 1$ we have that the law $\mathbb{P}^{(k)}_{0,N}$  converges weakly  on the path space $C([0,T],\mathbb{R}^{3k})$ to $\mathbb{P}_0^{\otimes k}$.
\end{theorem}
\begin{proof}
The proof can be found in \cite{ADU}.
\end{proof}

\section*{Acknowledgments}

The first and second authors would like to thank the Department of Mathematics, Universit\`a degli Studi di Milano for the warm hospitality. The second author is funded by the DFG under Germany’s Excellence Strategy - GZ 2047/1, project-id 390685813.

\bibliographystyle{plain}

\bibliography{ergodic3}

\end{document}